\documentclass[letterpaper,10pt]{article}

\usepackage[margin=0.5in]{geometry}
\usepackage[pagewise]{lineno}

\usepackage[hang,small,bf]{caption}
\usepackage{fullpage}
\usepackage{enumerate}
\usepackage{authblk}
\usepackage[colorinlistoftodos]{todonotes}
\usepackage{verbatim}%comments out anything between \begin{comment} and \end{comment}
\usepackage{section, amsthm, textcase, setspace, amssymb, lineno, amsmath, amssymb, amsfonts, latexsym, fancyhdr, longtable, ulem, mathtools}
\usepackage{epsfig, graphicx, pstricks,pst-grad,pst-text,tikz,colortbl}
\usepackage{epsf}
\usepackage{graphicx, color}
\usepackage{float}
\usepackage[rflt]{floatflt}
\usepackage[most]{tcolorbox}
\usepackage{amsfonts}
\usepackage{latexsym,enumitem}
\usetikzlibrary{fit,matrix,positioning}
\usepackage{pdflscape}
\usetikzlibrary{decorations.pathreplacing}
\usepackage{mathrsfs}
\usepackage{makecell}
\usetikzlibrary{decorations.markings}
\usepackage{tikz}
\usetikzlibrary{decorations.markings}
\usepackage{yfonts}
\usepackage{faktor}

\definecolor{darkgreen}{rgb}{0, 0.5, 0}

\newtheorem{theorem}{Theorem}
\newtheorem{lemma}[theorem]{Lemma}
\newtheorem{corollary}[theorem]{Corollary}

\newtheorem{definition}[theorem]{Definition}

\newtheorem{Ex}[theorem]{Example}

\newtheorem*{theorem*}{Theorem}
\newtheorem{remark}[theorem]{Remark}
\newtheorem{innercustomthm}{Theorem}
\newenvironment{customthm}[1]
  {\renewcommand\theinnercustomthm{#1}\innercustomthm}
  {\endinnercustomthm}

\newcommand{\rn}[1]{{\color{red} #1}}
\newcommand{\mf}{\mathfrak}
\newcommand{\bn}[1]{{\color{blue} #1}}

\newcommand{\ind}{{\rm ind \hspace{.1cm}}}

\newcommand{\C}{\mathbb{C}}

\newcommand{\g}{\mathfrak{g}}
\newcommand{\ul}[1]{\underline{#1}}

\tikzstyle{vertex}=[circle, draw, inner sep=0pt, minimum size=13pt]
\newcommand{\vertex}{\node[vertex]}
\tikzstyle{svertex}=[circle, draw, inner sep=0pt, minimum size=8pt]
\newcommand{\svertex}{\node[svertex]}

\begin{document}

\title{Contact seaweeds II:  type C}

\author[*]{Vincent E. Coll, Jr.}
\author[**]{Nicholas Russoniello}

\affil[*]{\small{Department of Mathematics, Lehigh University, Bethlehem, PA 18015: vec208@lehigh.edu}}
\affil[**]{\small{Department of Mathematics, College of William \& Mary, Williamsburg, VA 23185: nrussoniello@wm.edu}}

\maketitle
\begin{abstract}
\noindent
This paper is a continuation of earlier work on the construction of contact forms on seaweed algebras.  In the prequel to this paper, we show that every index-one seaweed subalgebra of $A_{n-1}=\mathfrak{sl}(n)$ is contact by identifying contact forms that arise from Dougherty's framework.  We extend this result to include index-one seaweed subalgebras of $C_{n}=\mathfrak{sp}(2n)$. Our methods are graph-theoretic and combinatorial.

%A $2n+1$-dimensional contact Lie algebra is one which admits a one-form $\varphi$ such that $\varphi \wedge (d\varphi)^n\ne0$.  Such algebras have index one -- though this property is not characteristic.  Here, we show that seaweed algebras of types $A$ and $C$ are contact precisely when they have index one. 
\end{abstract}



%\linenumbers
%\tableofcontents
\section{Introduction}

The index\footnote{The $index$ of a Lie algebra $(\g, [-,-])$ is defined by 
$\ind \g=\min_{\varphi\in \mathfrak{g}^*} \dim  (\ker (B_\varphi))$, where $\varphi$ is an element of the linear dual $\mathfrak{g}^*$  and $B_\varphi$ is the associated skew-symmetric \textit{Kirillov form} defined by 
$B_\varphi(x,y)=\varphi([x,y]), \textit{ for all }~ x,y\in\g.$} of a Lie algebra is an important algebraic invariant and recent investigations have concentrated on methods by which this integral invariant may be computed for a class of algebras called \textit{seaweed} (or \textit{biparabolic}) algebras (see \textbf{\cite{ACam}, \cite{DK}, \cite{Joseph},} and \textbf{\cite{Pan}}).  Seaweed algebras of classical type are certain Lie subalgbras of $\mathfrak{gl}(n)$\footnote{There is a basis-free definition of seaweed Lie algebras, but we do not require it for our current investigation.}; 
for example, a type-C ($C_{n}=\mathfrak{sp}(2n,\mathbb{C})$) seaweed is defined by two partial compositions of $n$ and a symmetry condition across the anti-diagonal. Ongoing, we will assume that all Lie algebras are over $\mathbb{C}$.

From the differential point of view, the problem of identifying which Lie groups admit a left-invariant contact structure, i.e., have associated Lie algebras that are contact, remains an open question (see \textbf{\cite{defcontact}, \cite{contactposet}, \cite{Diatta}, \cite{GR}} and \textbf{\cite{contacttoral}}).
Recall that a $(2k+1)$-dimensional Lie algebra $\g$ is called \textit{contact} if there exists a one-form $\varphi\in\g^*$ such that $\varphi  \wedge (d\varphi)^k\ne 0$.  In this case, $\mathfrak{g}$ is said to be $\textit{contact}$, $\varphi$ is called a \textit{contact form}, and $\varphi  \wedge (d\varphi)^k$ is a \textit{volume form} on the underlying Lie group. Contact Lie algebras have index one, but not characteristically so.

%fundamental work by Gromov has established that every odd-dimensional Lie group maintains a contact structure -- although 

As in the prequel (\textbf{\cite{contactA}}, 2022), we are concerned with constructing contact one-forms, when possible, on index-one seaweeds.  In \textbf{\cite{contactA}}, the authors establish 
that index-one, type-A seaweeds are necessarily contact by leveraging the ``Dougherty construction" (\textbf{\cite{Adiss}}, 2019; cf. \textbf{\cite{aria}}, 2020), which yields index-realizing one-forms with readily describable kernels -- compare with those of the Kostant one-forms. Here, we extend the methods used to procure the type-A result to include the more complicated type-C case, thus obtaining the following more general theorem, which constitutes the main result of this paper.

\begin{customthm}{\ref{thm:main}} If $\mathfrak{g}$ is a type-A or type-C seaweed, then $\mathfrak{g}$ is contact if and only if $\ind\g=1.$
\end{customthm}

\begin{remark}
It was recently shown in \textup{(\textbf{\cite{contactclass}}, 2023)} that an index-one seaweed is contact if and only if it is quasi-reductive. Combining this with a result of Panyushev \textup{(\textbf{\cite{PanRais}}, 2005)} which ensures each type-A and type-C seaweed is quasi-reductive yields Theorem~\ref{thm:main} above. However, the methods of \textup{\textbf{\cite{contactclass}}} are algebraic and existential, rather than combinatorial and constructive. Thus, in tandem with the type-A result from \textup{\textbf{\cite{contactA}}}, we provide an alternative proof of Theorem~\ref{thm:main}, along with an explicit, combinatorially described contact form for each type-A and type-C seaweed.
\end{remark}

\bigskip
\noindent
%Here, we seek to classify contact algebras among a class of matrix algebras called \textit{seaweed algebras} (``seaweeds'').  (which are of topical interest REFERNCVES  These algebras, along with their evocative name,  were first introduced by Dergachev and A. Kirillov in (\textbf{\cite{DK}}, 2000), where they defined such algebras as subalgebras of $\mathfrak{gl}(n)$ preserving certain flags of subspaces developed from two compositions of $n$. The passage to seaweeds of ``classical type" is accomplished by requiring that elements of the seaweed subalgebra of $\mathfrak{gl}(n)$ satisfy additional algebraic conditions. In particular, the Type-$A$ case ($A_{n-1}=\mathfrak{sl}(n)$) is defined by a vanishing trace condition, and the Type-C case ($C_n=\mathfrak{sp}(2n)$) is defined by appropriate (skew) symmetry across the antidiagonal. It was recently shown in \textbf{\cite{contactA}} that type-A seaweeds are contact exactly when their index is equal to one. Here we prove the same result for type-C seaweeds using certain regular one-forms constructed in \textbf{\cite{aria}}. The proof technique used in this article extends to all other classical types, so long as the regular one-forms are defined appropriately; however, such one-forms are difficult to construct in general.

The structure of the paper is as follows. In Section~\ref{sec:seaweeds}, we define seaweed algebras of type C and detail the construction of an associated planar graph: the seaweed's \textit{meander}.  Using a combinatorial formula based on the  number and type of connected components in the meander, one readily obtains the index of the corresponding seaweed. A collection of deterministic, graph-theoretic moves yield a ``winding-down" procedure for a meander, revealing the meander's \textit{homotopy type}. Index-one seaweeds are then classified according to their meander's homotopy type.  In Section~\ref{sec:regular}, we introduce another planar graph built from the seaweed's meander -- the \textit{component meander} of the seaweed  -- which is used to identify the \textit{core} and \textit{peak} of the seaweed. These  distinguished sets are certain sets of matrix locations, and they provide the grist for the Dougherty construction. Remarkably, when restricted to index-one, type-C seaweeds, the Dougherty construction yields regular forms that are also contact. Section~\ref{sec:main} consists of the proof of the main result, which is then used -- along with a convenient index formula for certain type-C seaweeds -- in Section~\ref{sec:examples} to quickly construct infinite families of contact, type-C seaweeds. %Section~\ref{sec:announcement} contains the announcement of a forthcoming result regarding the complete classification of contact seaweed subalgebras of reductive Lie algebras.

\section{Seaweeds and meanders}\label{sec:seaweeds}
\noindent
A seaweed subalgebra $\g$ of 
$\mathfrak{gl}(n)$ is defined by two compositions of $n$. The passage to seaweeds  of ``classical type" is realized by requiring that elements of $\g$ satisfy additional algebraic conditions. For example, the type-A case ($A_{n-1}=\mathfrak{sl}(n)$) is defined by a vanishing trace condition.  A type-C seaweed in its ``standard (matrix) form'' is a Lie subalgebra of $\mf{sp}(2n)$ and is
parametrized by two \textit{partial compositions} of $n$ as follows.  Let $$\ul{a}=(a_1,a_2,\dots ,a_m)\text{ and } \ul{b}=(b_1,b_2,\dots ,b_t) 
\text{ with } \displaystyle \sum_{i=1}^{m}a_i \leq n \text{ and }
\displaystyle \sum_{i=1}^{t}b_i \leq n,$$
%and assume, without loss of generality, that $\sum a_i \geq \sum b_i$.  
%By construction, the sequence of numbers in $\ul{a}$ determine the heights of triangles below the main diagonal in $\mf{p}_n^\A(\ul{a} \dd \ul{b})$ which may have nonzero entries, and the sequence of numbers in $\ul{b}$ determine the heights of triangles above the main diagonal. 
and let $\mathscr{L}$ and 
$\mathscr{U}$ denote, respectively, the subspaces of $M_{2n}(\mathbb{C})$ consisting of all lower and upper triangular matrices, respectively. If $X_{ \ul{a} }$ is the vector space of block-diagonal matrices whose blocks  have sizes 
\begin{eqnarray}
a_1\times a_1,\dots,a_m\times a_m, 2\left(n-\sum a_i\right) \times 2\left(n-\sum a_i\right), a_m \times a_m, \dots, a_1 \times a_1
\end{eqnarray}
and similarly for $X_{\ul{b}}$, then the underlying vector space for a type-C seaweed is the subspace of $M_{2n}(\mathbb{C})$ given by 
$$(X_{\ul{a}} \cap \mathscr{L}) ~\cup~ (X_{\ul{b}} \cap \mathscr{U}).$$
%whose underlying vector space is contained in  
%as follows.  To set the notation, let $\mathscr{L}$ and $\mathscr{U}$ denote, respectively, the lower and upper triangular matrices in $M_n(\CC)$.  Now, fix two ordered compositions of $n$, $\ul{a}=(a_1,\dots,a_m)$ and $\ul{b}=(b_1,\dots,b_t)$. Next, let $X_{\ul{a}}$ be the set of block-diagonal matrices whose blocks have sizes 
%A type-C seaweed is the subalgebra of $\mf{gl}(2n)$ spanned by the intersection of $X_{\ul{a}}$ with the lower triangular matrices, the intersection of $X_{\ul{b}}$ with the upper triangular matrices, and all diagonal matrices  -- along with 
%the algebraic symmetry conditions imposed by Theorem \ref{classification}, and is 
We denote a type-C seaweed with such a standard form by

\begin{eqnarray}\label{num}
\mathfrak{p}_{2n}^C\frac{a_1|\dots|a_m|2n-2\sum_{i=1}^ma_i|a_m|\dots|a_1}{b_1|\dots|b_t|2n-2\sum_{j=1}^tb_j|b_t|\dots|b_1}
\end{eqnarray}
or, more succinctly, by $\mathfrak{p}_{2n}^C\frac{a_1|\dots|a_m}{b_1|\dots|b_t}$, because of the evidently palindromic pattern of the numerator and denominator of the fraction in (\ref{num}).  See Example \ref{matrix}.

\begin{Ex}\label{matrix}
Consider the seaweed $\mathfrak{p}_{16}^C\frac{2|3}{1|6}$, whose standard form is illustrated in Figure \ref{matrix}.
The asterisks represent ``admissible'' locations of the seaweed where possible nonzero entries from $\C$ may appear; blank locations are forced zeroes.

\begin{figure}[H]\label{matrix1}
$$\begin{tikzpicture}[scale=0.4]
    \def\Node{\node [circle, fill, inner sep=3pt]}
	\draw (0,0) -- (0,16);
	\draw (0,16) -- (16,16);
	\draw (16,16) -- (16,0);
	\draw (16,0) -- (0,0);
	\draw [line width=3](0,16) -- (1,16);
	\draw [line width=3](1,16) -- (1,15);
	\draw [line width=3](1,15) -- (7,15);
	\draw [line width=3](7,15) -- (7,9);
	\draw [line width=3](7,9) -- (9,9);
	\draw [line width=3](9,9) -- (9,7);
    \draw [line width=3](9,7) -- (15,7);
    \draw [line width=3](15,7) -- (15,1);
    \draw [line width=3](15,1) -- (16,1);
    \draw [line width=3](16,1) -- (16,0);
    \draw [line width=3](0,16) -- (0,14);
    \draw [line width=3](0,14) -- (2,14);
    \draw [line width=3](2,14) -- (2,11);
    \draw [line width=3](2,11) -- (5,11);
    \draw [line width=3](5,11) -- (5,5);
    \draw [line width=3](5,5) -- (11,5);
    \draw [line width=3](11,5) -- (11,2);
    \draw [line width=3](11,2) -- (14,2);
    \draw [line width=3](14,2) -- (14,0);
    \draw [line width=3](14,0) -- (16,0);
	\draw [dotted] (0,16) -- (16,0);
	\draw [dotted] (0,0) -- (16,16);
	\draw [dotted] (0,8) -- (16,8);
	\draw [dotted] (8,0) -- (8,16);
	\node at (0.5,15.3) {{\LARGE *}};
	\node at (0.5,14.3) {{\LARGE *}};
	\node at (1.5,14.3) {{\LARGE *}};
	\node at (2.5,14.3) {{\LARGE *}};
	\node at (3.5,14.3) {{\LARGE *}};
	\node at (4.5,14.3) {{\LARGE *}};
	\node at (5.5,14.3) {{\LARGE *}};
	\node at (6.5,14.3) {{\LARGE *}};
	\node at (2.5,13.3) {{\LARGE *}};
	\node at (3.5,13.3) {{\LARGE *}};
	\node at (4.5,13.3) {{\LARGE *}};
	\node at (5.5,13.3) {{\LARGE *}};
	\node at (6.5,13.3) {{\LARGE *}};
	\node at (2.5,12.3) {{\LARGE *}};
	\node at (3.5,12.3) {{\LARGE *}};
	\node at (4.5,12.3) {{\LARGE *}};
	\node at (5.5,12.3) {{\LARGE *}};
	\node at (6.5,12.3) {{\LARGE *}};
	\node at (2.5,11.3) {{\LARGE *}};
	\node at (3.5,11.3) {{\LARGE *}};
	\node at (4.5,11.3) {{\LARGE *}};
	\node at (5.5,11.3) {{\LARGE *}};
	\node at (6.5,11.3) {{\LARGE *}};
	\node at (5.5,10.3) {{\LARGE *}};
	\node at (6.5,10.3) {{\LARGE *}};
	\node at (5.5,9.3) {{\LARGE *}};
	\node at (6.5,9.3) {{\LARGE *}};
	\node at (5.5,8.3) {{\LARGE *}};
	\node at (6.5,8.3) {{\LARGE *}};
	\node at (7.5,8.3) {{\LARGE *}};
	\node at (8.5,8.3) {{\LARGE *}};
	\node at (5.5,7.3) {{\LARGE *}};
	\node at (6.5,7.3) {{\LARGE *}};
	\node at (7.5,7.3) {{\LARGE *}};
	\node at (8.5,7.3) {{\LARGE *}};
	\node at (5.5,6.3) {{\LARGE *}};
	\node at (6.5,6.3) {{\LARGE *}};
	\node at (7.5,6.3) {{\LARGE *}};
	\node at (8.5,6.3) {{\LARGE *}};
	\node at (9.5,6.3) {{\LARGE *}};
	\node at (10.5,6.3) {{\LARGE *}};
	\node at (11.5,6.3) {{\LARGE *}};
	\node at (12.5,6.3) {{\LARGE *}};
	\node at (13.5,6.3) {{\LARGE *}};
	\node at (14.5,6.3) {{\LARGE *}};
	\node at (5.5,5.3) {{\LARGE *}};
	\node at (6.5,5.3) {{\LARGE *}};
	\node at (7.5,5.3) {{\LARGE *}};
	\node at (8.5,5.3) {{\LARGE *}};
	\node at (9.5,5.3) {{\LARGE *}};
	\node at (10.5,5.3) {{\LARGE *}};
	\node at (11.5,5.3) {{\LARGE *}};
	\node at (12.5,5.3) {{\LARGE *}};
	\node at (13.5,5.3) {{\LARGE *}};
	\node at (14.5,5.3) {{\LARGE *}};
	\node at (11.5,4.3) {{\LARGE *}};
	\node at (12.5,4.3) {{\LARGE *}};
	\node at (13.5,4.3) {{\LARGE *}};
	\node at (14.5,4.3) {{\LARGE *}};
	\node at (11.5,3.3) {{\LARGE *}};
	\node at (12.5,3.3) {{\LARGE *}};
	\node at (13.5,3.3) {{\LARGE *}};
	\node at (14.5,3.3) {{\LARGE *}};
	\node at (11.5,2.3) {{\LARGE *}};
	\node at (12.5,2.3) {{\LARGE *}};
	\node at (13.5,2.3) {{\LARGE *}};
	\node at (14.5,2.3) {{\LARGE *}};
	\node at (14.5,1.3) {{\LARGE *}};
    \node at (14.5,0.3) {{\LARGE *}};
    \node at (15.5,0.3) {{\LARGE *}};
    
\node[label=above:{1}] at (0.5,15.8) {};
\node[label=above:{6}] at (4,14.8) {};
\node[label=above:{2}] at (8,8.8){};
\node[label=above:{6}] at (12,6.8){};
\node[label=above:{1}] at (15.5,0.8){};
\node[label=left:{2}] at (0.2,15){};
\node[label=left:{3}] at (2.2,12.5){};
\node[label=left:{6}] at (5.2,8){};
\node[label=left:{3}] at (11.2,3.5){};
\node[label=left:{2}] at (14.2,1){};
\end{tikzpicture}$$
\caption{The type-C seaweed $\mathfrak{p}_{16}^C\frac{2|3}{1|6}$}
    \label{exC1}
\end{figure}
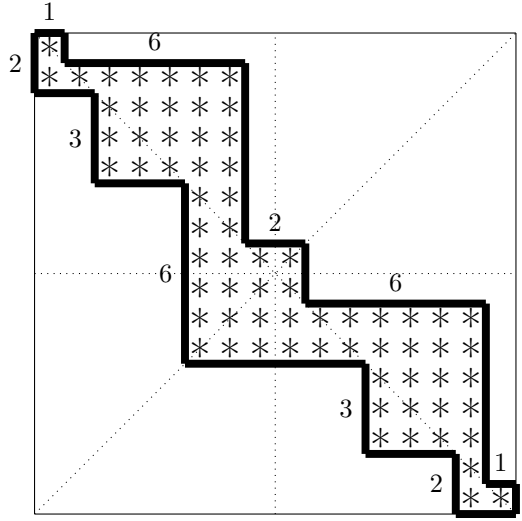
\end{Ex}

\begin{remark}\label{rem:Cchev}
We tacitly assume a standard \textup(Chevalley\textup) basis for a seaweed $\mathfrak{g}=\mathfrak{p}_{2n}^C\frac{a_1|\dots|a_m}{b_1|\dots|b_t}$ given by the union of the following sets of matrix units:
\begin{itemize}
    \item $\{e_{i,j}-e_{2n-j+1,2n-i+1}~|~1\leq i,j\leq n\},$
    \item $\{e_{i,j}~|~i+j=2n+1\},$
    \item $\{e_{i,j}+e_{2n-j+1,2n-i+1}~|~1\leq i\leq n,\text{ }n+1\leq j\leq 2n,\text{ and }i+j<2n+1\},$ and
    \item $\{e_{i,j}+e_{2n-j+1,2n-i+1}~|~n+1\leq i\leq 2n,\text{ }1\leq j\leq n,\text{ and }i+j<2n+1\},$
\end{itemize}
for all admissible locations $(i,j).$
\end{remark}

%%%%%%%%%%%%%%%%%%%%%%%%%%%%%%%%%%
%\subsection{Type-C Meanders}\label{construction}
To each type-C seaweed $\g=\mathfrak{p}_{2n}^C\frac{a_1|\dots|a_m}{b_1|\dots|b_t}$ we associate a planar graph called a (\textit{type-C}) \textit{meander}, denoted by $M^C_n(\g)$ and constructed as follows.  Place $n$ vertices $v_1$ through $v_n$ in a horizontal line, and create two partitions \textup(top and bottom\textup) of the vertices based on the given partial compositions of $n$. Following the type-A construction, draw arcs in the first $m$ top blocks and the first $t$ bottom blocks.   There may be vertices left over. Now define the following sets: $T_a=\{v_i\;|\;\sum_{j=1}^ma_j<i<n\}$ and $T_b=\{v_i\;|\;\sum_{j=1}^tb_j<i<n\}$. The set $T_n^C(\g)=(T_a\cup T_b)\backslash(T_a\cap T_b)$ is the \textit{tail} of $M_n^C(\g)$, and the \textit{aftertail} $\widetilde{T}_n^C(\g)$ of $M_n^C(\g)$ is $T_a\cap T_b$. By a slight  abuse of terminology, we also refer to $T_n^C(\g)$ and $\widetilde{T}_n^C(\g)$ as the tail and aftertail of $\g.$ See Example~\ref{ex:exC}.

\begin{Ex}\label{ex:exC} Consider the seaweed  
 $\g=\mathfrak{p}_{16}^C\frac{2|3}{1|6}$ whose meander $M_8^{C}(\g)$ is illustrated in 
Figure~\ref{fig:Chalfmeander}. Note that $T_a=\{v_6,v_7,v_8\},$ $T_b=\{v_8\},$ $T_8^C(\g)=\{v_6,v_7\},$ and $\widetilde{T}_8^C(\g)=\{v_8\}.$ The vertices in $T_n^C({\g})$ and $\widetilde{T}_n^C({\g})$ are colored blue and red, respectively -- a convention we employ ongoing.

\begin{figure}[H]
$$\begin{tikzpicture}[scale=0.75]
   \def\Node{\node [circle, fill, inner sep=2pt]}
  \vertex (1) at (0,0){1};
  \vertex (2) at (1,0){2};
  \vertex (3) at (2,0){3};
  \vertex (4) at (3,0){4};
  \vertex (5) at (4,0){5};
  \vertex[fill=blue, fill opacity=0.5, text opacity=1] (6) at (5,0){6};
  \vertex[fill=blue, fill opacity=0.5, text opacity=1] (7) at (6,0){7};
  \vertex[fill=red, fill opacity=0.5, text opacity=1] (8) at (7,0){8};
  \draw (2) to[bend right=50] (1);
  \draw (5) to[bend right=50] (3);
  \draw (2) to[bend right=50] (7);
  \draw (3) to[bend right=50] (6);
  \draw (4) to[bend right=50] (5);
\end{tikzpicture}$$
\caption{The meander ${M}_8^C(\g)$}
\label{fig:Chalfmeander}
\end{figure}
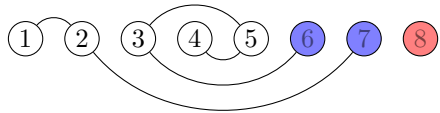
\end{Ex}

The identification of the tail in the meander leads to a combinatorial index formula for type-C seaweeds. In particular, we have the following result. 

\begin{theorem}[Coll et al. \textbf{\cite{indC}}, 2017]\label{thm:indexC}
If $\g=\mathfrak{p}_{2n}^C \frac{a_1|\dots|a_m}{b_1|\dots|b_t}$ is a type-C seaweed, then 
\begin{eqnarray}\label{indexformula}
\ind \g=2C+\tilde{P},
\end{eqnarray}

\noindent
where $C$ is the number of cycles and $\tilde{P}$ is the number of paths with zero or two endpoints in $T_n^C(\g)$.
\end{theorem}

\begin{remark}
Note that the aftertail is not involved in \textup{(\ref{indexformula})}; it is, however, essential in the construction of regular one-forms on a type-C seaweed \textup(see Section~\ref{sec:Cframework}\textup).
\end{remark}

\begin{Ex}\label{ex:indexC}
Consider $\g=\mathfrak{p}_{36}^C\frac{5|10}{2|4|3|1|1}$ with its meander $M_{18}^C(\g)$ illustrated in Figure~\ref{fig:indexC}. Notice that 
\begin{itemize}
    \item $\{v_1,v_2,v_4,v_5\}$ defines a cycle, 
    \item $\{v_{10},v_{11}\},\{v_{16}\},\{v_{17}\},$ and $\{v_{18}\}$ each define paths with zero endpoints in the tail, 
    \item $\{v_{13},v_8\}$ and $\{v_{15},v_6,v_3\}$ each define paths with exactly one endpoint in the tail, and
    \item $\{v_{14},v_7,v_9,v_{12}\}$ defines a path with two endpoints in the tail.
\end{itemize}
It now follows from Theorem~\ref{thm:indexC} that $\ind\g=2(1)+4+0+1=7$.

\begin{figure}[H]
$$\begin{tikzpicture}[scale=0.7]
    \def\Node{\node [circle, fill, inner sep=1.5pt]}
    \vertex (1) at (1,0){1};
    \vertex (2) at (2,0){2};
    \vertex (3) at (3,0){3};
    \vertex (4) at (4,0){4};
    \vertex (5) at (5,0){5};
    \vertex (6) at (6,0){6};
    \vertex (7) at (7,0){7};
    \vertex (8) at (8,0){8};
    \vertex (9) at (9,0){9};
    \vertex (10) at (10,0){10};
    \vertex (11) at (11,0){11};
    \vertex [fill=blue, fill opacity=0.5, text opacity=1] (12) at (12,0){12};
    \vertex [fill=blue, fill opacity=0.5, text opacity=1] (13) at (13,0){13};
    \vertex [fill=blue, fill opacity=0.5, text opacity=1] (14) at (14,0){14};
    \vertex [fill=blue, fill opacity=0.5, text opacity=1] (15) at (15,0){15};
    \vertex [fill=red, fill opacity=0.5, text opacity=1] (16) at (16,0){16};
    \vertex [fill=red, fill opacity=0.5, text opacity=1] (17) at (17,0){17};
    \vertex [fill=red, fill opacity=0.5, text opacity=1] (18) at (18,0){18};
    \draw (1) to[bend left=50] (5);
    \draw (2) to[bend left=50] (4);
    \draw (6) to[bend left=50] (15);
    \draw (7) to[bend left=50] (14);
    \draw (8) to[bend left=50] (13);
    \draw (9) to[bend left=50] (12);
    \draw (10) to[bend left=50] (11);
    
    \draw (1) to[bend right=50] (2);
    \draw (3) to[bend right=50] (6);
    \draw (4) to[bend right=50] (5);
    \draw (7) to[bend right=50] (9);
\end{tikzpicture}$$
    \caption{Meander $M_{18}^C(\g)$}
    \label{fig:indexC}
\end{figure}
\end{Ex}

A meander $M_n^C(\g)$ can be  contracted via a set of winding moves, each of which is predicated on the defining partial compositions of the seaweed at the time of move application. The ``winding-down'' process, detailed in Lemma \ref{lem:windingC} below,  continues until one of the partial compositions is reduced to the empty composition.   This process corresponds precisely to the reduction algorithm given in (Panyushev \textbf{\cite{Pan}}, 2001; cf. \textbf{\cite{PanRais}}, 2005) for type C, but we state the graph-theoretic version of the reduction as a set of winding moves as in \textbf{\cite{contactA}}. The \textit{signature} of the meander is the sequence of moves required to transform $M_n^C(\g)$ to a meander defined by at least one empty composition.

%\begin{tcolorbox}[breakable, enhanced]
\begin{lemma}\label{lem:windingC}
Let $\g=\mathfrak{p}_{2n}^C\frac{a_1|\dots|a_m}{b_1|\dots|b_t}$ be a type-C seaweed with associated meander $M_n^C(\g)$ -- in this setting, we say the meander $M=M_n^C(\g)$ is of fractional form $\frac{a_1|\dots|a_m}{b_1|\dots|b_t}.$ Create a meander $M'$ by one of the following moves.
\begin{enumerate}
    \item {Pure Contraction \textup(P\textup):} If $a_1>2b_1$, then $M\mapsto M'$ of fractional form $\frac{(a_1-2b_1)|b_1|a_2|\cdots|a_m}{b_2|b_3|\cdots|b_t}$.
	\item {Block Elimination  \textup(B\textup):} If $a_1=2b_1$, then $M\mapsto M'$ of fractional form $\frac{b_1|a_2|\cdots|a_m}{b_2|b_3|\cdots|b_t}$.
    \item {Rotation Contraction \textup(R\textup):} If $b_1<a_1<2b_1$, then $M\mapsto M'$ of fractional form $\frac{b_1|a_2|\cdots|a_m}{(2b_1-a_1)|b_2|\cdots|b_t}$.
    \item {Component Deletion \textup(C\textup(c\textup)\textup):} If $a_1=b_1=c$, then $M\mapsto M'$ of fractional form $\frac{a_2|\cdots|a_m}{b_2|\cdots|b_t}$.
    \item {Flip \textup(F\textup):} If $a_1<b_1$, then $M\mapsto M'$ of fractional form $\frac{b_1|\cdots|b_t}{a_1|\cdots|a_m}$.
\end{enumerate}
  The five moves are called \textit{{winding-down moves}}.  For all moves, except the Component Deletion move, $\g$ and $\g'$ \textup(the seaweed with meander $M_k^C(\g')=M'$, $k\leq n$\textup) have the same index. 
\end{lemma}
%\end{tcolorbox}

\medskip
\noindent
Now, define the (\textit{type-C}) \textit{homotopy type}
%\footnote{Though defined slightly differently in \textbf{\cite{aria}}, the definition of type-C homotopy type is equivalent to Dougherty's notion of ``reduced homotopy type."} 
of $\g=\mathfrak{p}_{2n}^C\frac{a_1|\dots|a_m}{b_1|\dots|b_t}$ with meander $M_n^C(\g)$ to be $$\mathcal{H}_C(c_1,\dots,c_{q_1},\bn{c_{q_1+1},\dots,c_{q_2}},\rn{c_{q_2+1}}),$$ where $C(c_i),$ $1\leq i\leq c_{q_1},$ are the component deletion moves in the signature of $M_n^C(\g),$ the fully wound-down meander is $M_k^C(\g'),$ where $k\leq n$ and $\g'=\mathfrak{p}_{2k}^C\frac{c_{q_1+1}|\dots|c_{q_2}}{\emptyset},$ and $c_{q_2+1}=2\left(k-\sum_{j=q_1+1}^{q_2}c_j\right).$ Notice that the fully wound-down meander $M_k^C(\g')$ consists only of tail and aftertail vertices, and moreover, neither the tail nor the aftertail are affected by the winding-down moves. Thus, the tail and aftertail of $M_k^C(\g')$ are identical to those of $M_n^C(\g).$ See Example~\ref{ex:windingC}.

\begin{remark}
    It is worth noting that since type-A seaweeds are based on two full compositions of the same integer $n$, there is no tail or aftertail.  Even so, the winding moves of Lemma \ref{lem:windingC} still apply and the winding-down process terminates when the empty meander is reached. Note also that a type-C seaweed may be defined by two full compositions, in which case, all relevant results coincide with those of type-A seaweeds.
\end{remark}

 %and we are then reduced to the type-A case, where the meander of an index one seaweed must consist of a single cycle or two paths.
 
\begin{Ex}\label{ex:windingC}
The meander of fractional form $\frac{5|10}{2|4|3|1|1}$ of Example~\ref{ex:indexC} has signature $PFPFC(2)PPBC(1),$ and the resulting meander is of fractional form $\frac{3|1}{\emptyset},$ so the homotopy type is $\mathcal{H}_C(2,1,\bn{3,1},\rn{6}).$ The winding-down moves associated with the signature are illustrated in Figure~\ref{fig:windingC}.

\begin{figure}[H]
$$\begin{tikzpicture}[scale=0.28]
    \def\Node{\node [circle, fill, inner sep=1.5pt]}
    \Node (1) at (1,0){};
    \Node (2) at (2,0){};
    \Node (3) at (3,0){};
    \Node (4) at (4,0){};
    \Node (5) at (5,0){};
    \Node (6) at (6,0){};
    \Node (7) at (7,0){};
    \Node (8) at (8,0){};
    \Node (9) at (9,0){};
    \Node (10) at (10,0){};
    \Node (11) at (11,0){};
    \Node [fill=blue] (12) at (12,0){};
    \Node [fill=blue] (13) at (13,0){};
    \Node [fill=blue] (14) at (14,0){};
    \Node [fill=blue] (15) at (15,0){};
    \Node [fill=red] (16) at (16,0){};
    \Node [fill=red] (17) at (17,0){};
    \Node [fill=red] (18) at (18,0){};
    \draw (1) to[bend left=50] (5);
    \draw (2) to[bend left=50] (4);
    \draw (6) to[bend left=50] (15);
    \draw (7) to[bend left=50] (14);
    \draw (8) to[bend left=50] (13);
    \draw (9) to[bend left=50] (12);
    \draw (10) to[bend left=50] (11);
    
    \draw (1) to[bend right=50] (2);
    \draw (3) to[bend right=50] (6);
    \draw (4) to[bend right=50] (5);
    \draw (7) to[bend right=50] (9);
    
    \node at (9,-2){$\frac{5|10}{2|4|3|1|1}$};
    
    \node at (20,0){$\overset{P}{\mapsto}$};
    
    \Node (19) at (22,0){};
    \Node (20) at (23,0){};
    \Node (21) at (24,0){};
    \Node (22) at (25,0){};
    \Node (23) at (26,0){};
    \Node (24) at (27,0){};
    \Node (25) at (28,0){};
    \Node (26) at (29,0){};
    \Node (27) at (30,0){};
    \Node [fill=blue] (28) at (31,0){};
    \Node [fill=blue] (29) at (32,0){};
    \Node [fill=blue] (30) at (33,0){};
    \Node [fill=blue] (31) at (34,0){};
    \Node [fill=red] (32) at (35,0){};
    \Node [fill=red] (33) at (36,0){};
    \Node [fill=red] (34) at (37,0){};
    \draw (20) to[bend left=50] (21);
    \draw (22) to[bend left=50] (31);
    \draw (23) to[bend left=50] (30);
    \draw (24) to[bend left=50] (29);
    \draw (25) to[bend left=50] (28);
    \draw (26) to[bend left=50] (27);
    
    \draw (19) to[bend right=50] (22);
    \draw (20) to[bend right=50] (21);
    \draw (23) to[bend right=50] (25);
    
    \node at (29,-2){$\frac{1|2|10}{4|3|1|1}$};
    
    \node at (39.5,0){$\overset{FPF}{\mapsto}$};
    
    \Node (35) at (41.5,0){};
    \Node (36) at (42.5,0){};
    \Node (37) at (43.5,0){};
    \Node (38) at (44.5,0){};
    \Node (39) at (45.5,0){};
    \Node (40) at (46.5,0){};
    \Node (41) at (47.5,0){};
    \Node (42) at (48.5,0){};
    \Node[blue] (43) at (49.5,0){};
    \Node[blue] (44) at (50.5,0){};
    \Node[blue] (45) at (51.5,0){};
    \Node[blue] (46) at (52.5,0){};
    \Node[red] (47) at (53.5,0){};
    \Node[red] (48) at (54.5,0){};
    \Node[red] (49) at (55.5,0){};
    \draw (35) to[bend left=50] (36);
    \draw (37) to[bend left=50] (46);
    \draw (38) to[bend left=50] (45);
    \draw (39) to[bend left=50] (44);
    \draw (40) to[bend left=50] (43);
    \draw (41) to[bend left=50] (42);
    
    \draw (35) to[bend right=50] (36);
    \draw (38) to[bend right=50] (40);
    
    \node at (48.5,-2){$\frac{2|10}{2|1|3|1|1}$};
    
\end{tikzpicture}$$
    \label{fake1}
\end{figure}

\begin{figure}[H]
$$\begin{tikzpicture}[scale=0.28]
    \node at (-1,0){$\overset{C(2)}{\mapsto}$};
    
    \def\Node{\node [circle, fill, inner sep=1.5pt]}
    \Node (1) at (1,0){};
    \Node (2) at (2,0){};
    \Node (3) at (3,0){};
    \Node (4) at (4,0){};
    \Node (5) at (5,0){};
    \Node (6) at (6,0){};
    \Node [fill=blue] (7) at (7,0){};
    \Node [fill=blue] (8) at (8,0){};
    \Node [fill=blue] (9) at (9,0){};
    \Node [fill=blue] (10) at (10,0){};
    \Node [fill=red] (11) at (11,0){};
    \Node [fill=red] (12) at (12,0){};
    \Node [fill=red] (13) at (13,0){};
    \draw (2) to[bend right=50] (4);
    
    \draw (1) to[bend left=50] (10);
    \draw (2) to[bend left=50] (9);
    \draw (3) to[bend left=50] (8);
    \draw (4) to[bend left=50] (7);
    \draw (5) to[bend left=50] (6);
    
    \node at (6.5,-2){$\frac{10}{1|3|1|1}$};
    
    \node at (15,0){$\overset{P}{\mapsto}$};
    
    \Node (14) at (17,0){};
    \Node (15) at (18,0){};
    \Node (16) at (19,0){};
    \Node (17) at (20,0){};
    \Node (18) at (21,0){};
    \Node [fill=blue] (19) at (22,0){};
    \Node [fill=blue] (20) at (23,0){};
    \Node [fill=blue] (21) at (24,0){};
    \Node [fill=blue] (22) at (25,0){};
    \Node [fill=red] (23) at (26,0){};
    \Node [fill=red] (24) at (27,0){};
    \Node [fill=red] (25) at (28,0){};
    \draw (14) to[bend right=50] (16);
    
    \draw (14) to[bend left=50] (21);
    \draw (15) to[bend left=50] (20);
    \draw (16) to[bend left=50] (19);
    \draw (17) to[bend left=50] (18);
    
    \node at (22,-2){$\frac{8|1}{3|1|1}$};

    \node at (30,0){$\overset{P}{\mapsto}$};
    
    \Node (26) at (32,0){};
    \Node (27) at (33,0){};
    \Node [fill=blue] (28) at (34,0){};
    \Node [fill=blue] (29) at (35,0){};
    \Node [fill=blue] (30) at (36,0){};
    \Node [fill=blue] (31) at (37,0){};
    \Node [fill=red] (32) at (38,0){};
    \Node [fill=red] (33) at (39,0){};
    \Node [fill=red] (34) at (40,0){};
    
    \draw (26) to[bend left=50] (27);
    \draw (28) to[bend left=50] (30);
    
    \node at (35.5,-2){$\frac{2|3|1}{1|1}$};

    \node at (42,0){$\overset{B}{\mapsto}$};
    
    \Node (35) at (44,0){};
    \Node [fill=blue] (36) at (45,0){};
    \Node [fill=blue] (37) at (46,0){};
    \Node [fill=blue] (38) at (47,0){};
    \Node [fill=blue] (39) at (48,0){};
    \Node [fill=red] (40) at (49,0){};
    \Node [fill=red] (41) at (50,0){};
    \Node [fill=red] (42) at (51,0){};
    
    \draw (36) to[bend left=50] (38);

    \node at (47,-2){$\frac{1|3|1}{1}$};
\end{tikzpicture}$$
\label{fake2}
\end{figure}

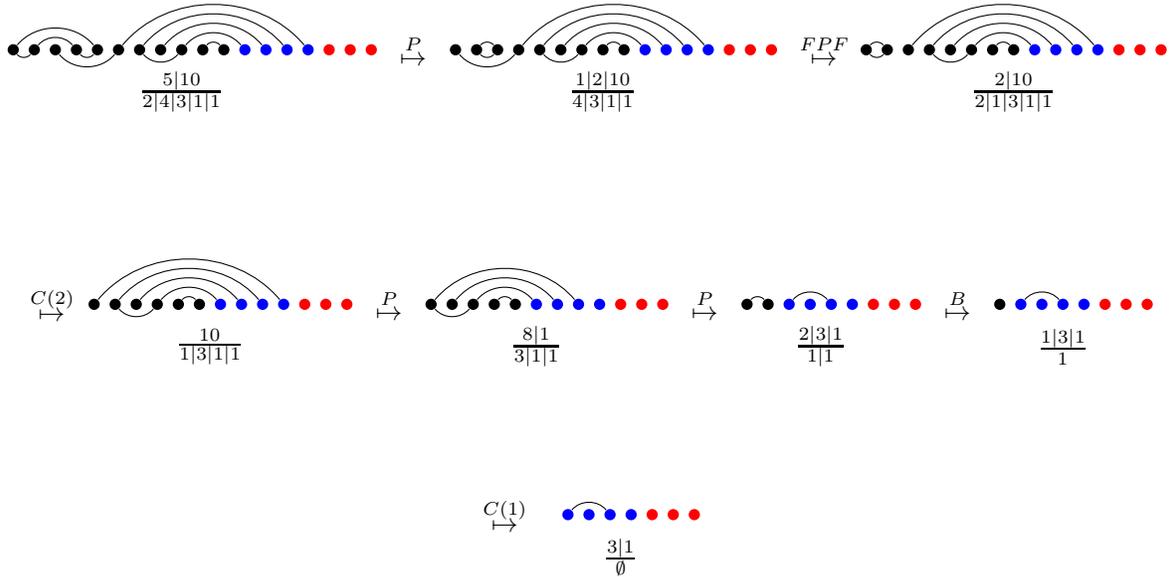
\begin{figure}[H]
$$\begin{tikzpicture}[scale=0.28]
    \node at (-2,0){$\overset{C(1)}{\mapsto}$};
    
    \def\Node{\node [circle, fill, inner sep=1.5pt]}
    \Node [fill=blue] (1) at (1,0){};
    \Node [fill=blue] (2) at (2,0){};
    \Node [fill=blue] (3) at (3,0){};
    \Node [fill=blue] (4) at (4,0){};
    \Node [fill=red] (5) at (5,0){};
    \Node [fill=red] (6) at (6,0){};
    \Node [fill=red] (7) at (7,0){};
    
    \draw (1) to[bend left=50] (3);
    
    \node at (3.5,-2){$\frac{3|1}{\emptyset}$};
\end{tikzpicture}$$
    \caption{Winding down the meander of fractional form $\frac{5|10}{2|4|3|1|1}$}
    \label{fig:windingC}
\end{figure}
\end{Ex}

The type-C homotopy type encodes essential information for the construction of regular one-forms in Section~\ref{sec:Cframework}. So although we may use Theorem~\ref{thm:indexC} to identify index-one, type-C seaweeds, it will be helpful to have the following index formula for a type-C seaweed in terms of its homotopy type.

\begin{theorem}[Dougherty \textbf{\cite{Adiss}}, Theorem 4.3.13]\label{thm:indhtC}
If $\g$ is a type-C seaweed with homotopy type $$\mathcal{H}_C(c_1,\dots,c_{q_1},\bn{c_{q_1+1},\dots,c_{q_2}},\rn{c_{q_2+1}}),$$ then $$\ind\g=\sum_{i=1}^{q_1}c_i+\sum_{i=q_1+1}^{q_2}\left\lfloor\frac{c_i}{2}\right\rfloor+\frac{c_{q_2+1}}{2}.$$
\end{theorem}

\begin{corollary}\label{cor:ind1htC}
A type-C seaweed $\g$ has index one if and only if it has one of the following homotopy types:
\begin{enumerate}
    \item[\textup{1a.}] $\mathcal{H}_C({\color{blue}1,\dots,1,2,1,\dots,1})$,
    \item[\textup{1b.}] $\mathcal{H}_C({\color{blue}1,\dots,1,3,1,\dots,1}),$
    \item[\textup{2.}] $\mathcal{H}_C(\color{blue}1,\dots,1,{\color{red}2}),$ or
    \item[\textup{3.}] $\mathcal{H}_C(1,{\color{blue}1,\dots,1})$.
\end{enumerate}
\end{corollary}

\noindent
Note that in the above cases, the sets of blue 1's are allowed to be empty.

%%%%%%%%%%%%%%%%%%%%%%%%%%%%%%%%%%%%%%
%%%%%%%%%%%%%%%%%%%%%%%%%%%%%%%%%%%%%%
\section{Framework for Regular Forms }\label{sec:regular}

%In this section, we briefly recall the notation and terminology of \textbf{\cite{contactA}} to describe the construction of regular one-forms on an index-one, type-C seaweed. Most of the construction follows \textit{mutatis mutandis} from the construction in type A. 

%Firstly, we identify a certain subset of the admissible locations $\textup(i,j\textup)$ of a given seaweed. With accommodations for parity the sum of the duals of these matrix units will yield a regular forms on the seaweed.  In Section 4, we establish that these regular forms are, in fact, contact forms.

In her dissertation \textbf{\cite{Adiss}}, Dougherty establishes an inductive framework for the explicit construction of multiple \textit{regular} (index-realizing) one-forms -- each having associated Kirillov forms with easily describable kernels -- on, in particular, type-A and type-C seaweeds (cf. \textbf{\cite{aria}}). On index-one, type-C seaweeds, we claim that a subset of such regular forms are, in fact, contact, and the proof of this claim is the subject of Section~\ref{sec:main}. First, we establish the \textit{Dougherty Framework} for type-C seaweeds, which requires some preliminary constructions.

%In Section \ref{sec:Aframework}, we describe the construction in type C. (For a type-A example, we refer the reader to XXXX.)  

\subsection{The Component Meander}
Recall that the meander of the seaweed $\g=\mathfrak{p}_{2n}^C\frac{a_1|\dots|a_m}{b_1|\dots|b_t},$ where $(a_1,\dots,a_m)$ and $(b_1,\dots,b_m)$ are partial compositions of $n$, has $n$ vertices. Also recall that $\g$ may be written as $$\g=\mathfrak{p}_{2n}^C\frac{a_1|\dots|a_m|2\left(n-\sum_{i=1}^ma_i\right)|a_m|\dots|a_1}{b_1|\dots|b_t|2\left(n-\sum_{i=1}^tb_i\right)|b_t|\dots|b_1},$$ so define the meander with $2n$ vertices and fractional form $\frac{a_1|\dots|a_m|2(n-\sum_{i=1}^ma_i)|a_m|\dots|a_1}{b_1|\dots|b_t|2(n-\sum_{i=1}^tb_i)|b_t|\dots|b_1}$ to be the \textit{full meander} $M_{2n}^C(\g).$ Placing the standard (counterclockwise) orientation on $M_{2n}^C(\g)$ yields the \textit{oriented full meander} $\overrightarrow{M}_{2n}^C(\g).$ Note that $M_{2n}^C(\g)$ is constructed using two ``full" compositions of $2n,$ and graphically, the full meander is two horizontally mirrored copies of $M_n^C(\g)$ with additional edges incident to the tail and aftertail vertices of $M_n^C(\g).$ See Example~\ref{ex:fullmeander1}.

%The meander for a type-C seaweed subalgebra of $\mathfrak{sp}(2n)$ only requires $n$ vertices.  
%Using the symmetry conditions attendant to a type-C seaweed, one can also construct an associated planer graph using all $2n$ vertices -- so two full compositions of $2n$ are in play -- and build the meander as before.  This produces what we call the \textit{full meander} associated to the seaweed $\g$ and write  $M_{2 n}^C(\g)$, or $\overrightarrow{M}_{2n}^C(\g)$ for the \textit{oriented full meander}, which has directed edges oriented counterclockwise. Graphically, we append a mirror image of  $M_{ n}^C(\g)$ to  $M_{n}^C(\g)$ to yield $M_{2 n}^C(\g)$. See Example \ref{ex:fullmeander1}.  Note that we needed $M_{ n}^C(\g)$ to identify the tail and aftertail vertices.   

\begin{Ex}\label{ex:fullmeander1}
Consider the seaweed $\g=\mathfrak{p}_{18}^C\frac{5|10}{2|4|3|1|1}$
whose full meander $M_{2 n}^C(\g)$ is illustrated in Figure \ref{fig:fullmeander1}. We maintain the established coloring scheme and similarly color the edges in $M_{2 n}^C(\g)$ according to whether they participate in a ``tail component" or an ``aftertail component", i.e., a component containing a tail vertex or an aftertail vertex, respectively.

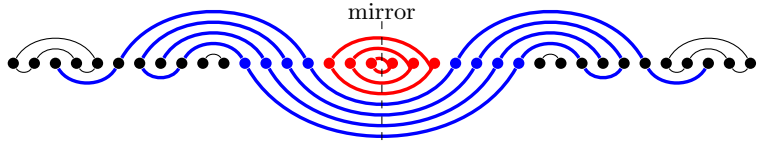
\begin{figure}[H]
$$\begin{tikzpicture}[scale=0.28]
    \def\Node{\node [circle, fill, inner sep=1.5pt]}
    \Node (1) at (1,0){};
    \Node (2) at (2,0){};
    \Node (3) at (3,0){};
    \Node (4) at (4,0){};
    \Node (5) at (5,0){};
    \Node (6) at (6,0){};
    \Node (7) at (7,0){};
    \Node (8) at (8,0){};
    \Node (9) at (9,0){};
    \Node (10) at (10,0){};
    \Node (11) at (11,0){};
    \Node [blue] (12) at (12,0){};
    \Node [blue] (13) at (13,0){};
    \Node [blue] (14) at (14,0){};
    \Node [blue] (15) at (15,0){};
    \Node [red] (16) at (16,0){};
    \Node [red] (17) at (17,0){};
    \Node [red] (18) at (18,0){};
    \Node [red] (1800) at (19,0){};
    \Node [red] (1700) at (20,0){};
    \Node [red] (1600) at (21,0){};
    \Node [blue] (1500) at (22,0){};
    \Node [blue] (1400) at (23,0){};
    \Node [blue] (1300) at (24,0){};
    \Node [blue] (1200) at (25,0){};
    \Node (1100) at (26,0){};
    \Node (1000) at (27,0){};
    \Node (900) at (28,0){};
    \Node (800) at (29,0){};
    \Node (700) at (30,0){};
    \Node (600) at (31,0){};
    \Node (500) at (32,0){};
    \Node (400) at (33,0){};
    \Node (300) at (34,0){};
    \Node (200) at (35,0){};
    \Node (100) at (36,0){};
    \draw (1) to[bend left=60] (5);
    \draw (2) to[bend left=60] (4);
    \draw [line width=0.45mm, blue] (6) to[bend left=60] (15);
    \draw [line width=0.45mm, blue] (7) to[bend left=60] (14);
    \draw [line width=0.45mm, blue] (8) to[bend left=60] (13);
    \draw [line width=0.45mm, blue] (9) to[bend left=60] (12);
    \draw (10) to[bend left=60] (11);
    
    \draw (1) to[bend right=60] (2);
    \draw [line width=0.45mm, blue] (3) to[bend right=60] (6);
    \draw (4) to[bend right=60] (5);
    \draw [line width=0.45mm, blue] (7) to[bend right=60] (9);
    
    \draw (100) to[bend right=60] (500);
    \draw (200) to[bend right=60] (400);
    
    \draw [line width=0.45mm, blue] (600) to[bend right=60] (1500);
    \draw [line width=0.45mm, blue] (700) to[bend right=60] (1400);
    \draw [line width=0.45mm, blue] (800) to[bend right=60] (1300);
    \draw [line width=0.45mm, blue] (900) to[bend right=60] (1200);
    \draw (1000) to[bend right=60] (1100);
    
    \draw (100) to[bend left=60] (200);
    \draw [line width=0.45mm, blue] (300) to[bend left=60] (600);
    \draw (400) to[bend left=60] (500);
    \draw [line width=0.45mm, blue] (700) to[bend left=60] (900);
    
    \draw [line width=0.45mm, blue] (12) to[bend right=60] (1200);
    \draw [line width=0.45mm, blue] (13) to[bend right=60] (1300);
    \draw [line width=0.45mm, blue] (14) to[bend right=60] (1400);
    \draw [line width=0.45mm, blue] (15) to[bend right=60] (1500);
    \draw [line width=0.45mm, red] (16) to[bend right=60] (1600) to[bend right=60] (16);
    \draw [line width=0.45mm, red] (17) to[bend right=60] (1700) to[bend right=60] (17);
    \draw [line width=0.45mm, red] (18) to[bend right=60] (1800) to[bend right=60] (18);
    
    \draw[dashed] (18.5,2)--(18.5,-4);
    
    \node at (18.5,2.5){\small mirror};
\end{tikzpicture}$$
\caption{The full meander $M_{2 n}^C(\g)$ }
\label{fig:fullmeander1}
\end{figure}
\end{Ex}

\medskip
For the next construction, we first note that the 
winding moves of Lemma~\ref{lem:windingC} apply to $M_{2n}^C(\g)$; that is, we may apply the winding moves to the meander of fractional form $\frac{a_1|\dots|a_m|2n-2\sum_{i=1}^ma_i|a_m|\dots|a_1}{b_1|\dots|b_t|2n-2\sum_{j=1}^tb_j|b_t|\dots|b_1}.$ The sequence of winding moves used to contract $M_{2n}^C(\g)$ is similarly called the \textit{signature} of $M_{2n}^C(\g),$ and we define the \textit{component meander} $CM(\g)$ of $\g$ as the meander with the same signature as $M_{2n}^C(\g)$ except that the parameter $c$ in all Component Deletion moves are modified so that $c=1$. The edges of $CM(\g)$ can be oriented counterclockwise to yield the 
\textit{oriented component meander}, which we denote by 
$\overrightarrow{CM}(\g)$.  See Example \ref{componentmeander}.

\begin{Ex}\label{componentmeander}
Consider, once again, the seaweed $\g=\mathfrak{p}_{18}^C\frac{5|10}{2|4|3|1|1}$
whose oriented component meander $\overrightarrow{CM}(\g)$ is illustrated in Figure \ref{fig:genCmeander}.
\end{Ex}

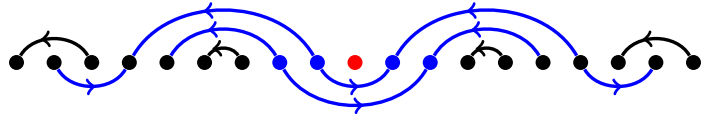
\begin{figure}[H]
$$\begin{tikzpicture}[scale=0.5]\label{colored}
\def\Node{\node [circle, fill, inner sep=2pt]}
\tikzset{->-/.style={decoration={
  markings,
  mark=at position .55 with {\arrow{>}}},postaction={decorate}}}
  \tikzset{-<-/.style={decoration={
  markings,
  mark=at position .45 with {\arrow{<}}},postaction={decorate}}}

\Node (1) at (0,0){};
\Node (2) at (1,0){};
\Node (3) at (2,0){};
\Node (4) at (3,0){};
\Node (5) at (4,0){};
\Node (6) at (5,0){};
\Node (7) at (6,0){};
\Node[blue] (8) at (7,0){};
\Node[blue] (9) at (8,0){};
\Node[red] (10) at (9,0){};
\Node[blue] (11) at (10,0){};
\Node[blue] (12) at (11,0){};
\Node (13) at (12,0){};
\Node (14) at (13,0){};
\Node (15) at (14,0){};
\Node (16) at (15,0){};
\Node (17) at (16,0){};
\Node (18) at (17,0){};
\Node (19) at (18,0){};

\draw[-<-, line width=0.45mm] (1) to[bend left=60] (3);
\draw[line width=0.45mm, color=blue,->-] (2) to[bend right=60] (4);
\draw[line width=0.45mm, color=blue,-<-] (4) to[bend left=60] (9);
\draw[line width=0.45mm, color=blue,-<-] (5) to[bend left=60] (8);
\draw[-<-, line width=0.45mm] (6) to[bend left=60] (7);
\draw[line width=0.45mm, color=blue,->-] (9) to[bend right=60] (11);
\draw[line width=0.45mm, color=blue,->-] (8) to[bend right=60] (12);
\draw[-<-, line width=0.45mm] (13) to[bend left=60] (14);
\draw[line width=0.45mm, color=blue,-<-] (12) to[bend left=60] (15);
\draw[line width=0.45mm, color=blue,-<-] (11) to[bend left=60] (16);
\draw[line width=0.45mm, color=blue,->-] (16) to[bend right=60] (18);
\draw[-<-, line width=0.45mm] (17) to[bend left=60] (19);
\end{tikzpicture}$$
\vspace{-2em}
\caption{The oriented component meander $\protect\overrightarrow{CM}(\g)$}
\label{fig:genCmeander}
\end{figure}

In Section~\ref{coreandpeak}, the component meander of a seaweed is used to define distinguished subsets of admissible locations of the seaweed called the ``core" and ``peak." These, in turn, allow for the identification of the summands of a regular one-form on the seaweed.

\subsection{The Core and Peak Sets}\label{coreandpeak}
In order to define the core and peak of $\g=\mathfrak{p}^C_{2n}\frac{a_1|\dots|a_m}{b_1|\dots|b_t}$, we first require a certain decomposition of the underlying vector space of $\mathfrak{g}.$ If $\mathfrak{g}$ has homotopy type $\mathcal{H}_C(c_1,\dots,c_{q_1},\bn{c_{q_1+1},\dots,c_{q_2}},\rn{c_{q_2+1}}),$ then the \textit{vector space decomposition} of $\g$ is
\begin{equation}\label{eqn:decomp}
\g=\bigoplus_{i=1}^{q_2+1}\g|_{c_i},
\end{equation}
where $\mathfrak{g}|_{c_i}$ is the subspace of $\g$ corresponding to a particular component in $M_{2 n}^C(\g)$ of size $c_i.$ Note that, in this notation, two distinct subspaces corresponding to components of equal size are indistinguishable, except if the components' interactions with the tail and aftertail differ. See Example~\ref{ex:vsdecompC}.

Define the sets 
$$V_{c_i}=\{I\;|\;v_{I}\text{ is a vertex on the path of $CM(\g)$ corresponding to $c_i$}\},$$

\vspace{-1.5em}

$$\textgoth{C}_{c_i}=\{I\times I\;|\;I\in V_{c_i}\}, \text{and}$$ 
$$\textgoth{P}_{c_i}=\{I\times J\;|\;I,J\in V_{c_i}\text{ and }(v_I,v_J) \text{ is an edge in } \overrightarrow{CM}(\g)\}.$$

\noindent The set $\textgoth{C}_{c_i}$ is the 
\textit{core set} of $\g|_{c_i}$ -- the set of $c_i\times c_i$ blocks on the diagonal of $\g$ contained in $\g|_{c_i}$ -- and $\textgoth{P}_{c_i}$ is the \textit{peak set} of $\g|_{c_i}.$

Now, define the \textit{core $\textgoth{C}_{\g}$ of $\g$} and the \textit{peak $\textgoth{P}_\g$ of $\g$} as the union of the core and peak sets, respectively, of the components in the vector space decomposition (\ref{eqn:decomp}); that is,

$$\textgoth{C}_\g=\bigcup_{i=1}^{q_2+1}\textgoth{C}_{c_i}\hspace{2em}\text{ and }\hspace{2em}\textgoth{P}_\g=\bigcup_{i=1}^{q_2+1}\textgoth{P}_{c_i}.$$

\noindent 
We refer to elements of $\textgoth{C}_{\g}$ and $\textgoth{P}_{\g}$ as \textit{core blocks} and \textit{peak blocks}, respectively.

\begin{Ex}\label{ex:vsdecompC}
Consider our running example type-C seaweed $\g=\g_{18}^C\frac{5|10}{2|4|3|1|1}$. Note that $\g$ has homotopy type $\mathcal{H}_C(2,1,\bn{3,1},\rn{6}),$ yielding the vector space decomposition $$\g=\g|_2\oplus\g|_1\oplus\g|_{\bn{3}}\oplus\g|_{\bn{1}}\oplus\g|_{\rn{6}}.$$ The subspaces $\g|_{c_i}$ are outlined and colored in Figure~\ref{fig:vspacedecompC}. Since each is contained in $\g\subset\mathfrak{sp}(18),$ the subspaces $\g|_{c_i}$ maintain the necessary symmetry conditions about the antidiagonal; for example, a basis for $\g|_1$ is given by $\{e_{10,10}-e_{27,27},e_{11,10}-e_{27,26},e_{11,11}-e_{26,26}\}.$
\begin{figure}[H]
$$\begin{tikzpicture}[scale=0.27]
    \def\Node{\node [circle, fill, inner sep=1.1pt]}
    \draw (0,0)--(36,0)--(36,36)--(0,36)--(0,0);
    \draw [line width=.65mm] (0,36)--(0,31)--(5,31)--(5,21)--(15,21)--(15,15)--(21,15)--(21,5)--(31,5)--(31,0)--(36,0)--(36,2)--(34,2)--(34,6)--(30,6)--(30,9)--(27,9)--(27,10)--(26,10)--(26,11)--(25,11)--(25,25)--(11,25)--(11,26)--(10,26)--(10,27)--(9,27)--(9,30)--(6,30)--(6,34)--(2,34)--(2,36)--(0,36);
    \draw [dotted] (0,0)--(36,36)--(0,36)--(36,0)--(18,0)--(18,36);
    \draw [dotted] (0,18)--(36,18);
    \draw [line width=.45mm] (5,31)--(5,33)--(2,33)--(2,34);
    \draw [line width=.45mm] (31,5)--(33,5)--(33,2)--(34,2);
    \draw [line width=.45mm] (6,30)--(6,22)--(22,22)--(22,6)--(30,6);
    \draw [line width=.45mm] (9,27)--(9,25)--(11,25);
    \draw [line width=.45mm] (27,9)--(25,9)--(25,11);
    \draw [line width=.45mm] (15,21)--(21,21)--(21,15);
    
    \draw [fill=blue, fill opacity=0.5] (2,34)--(6,34)--(6,30)--(9,30)--(9,25)--(25,25)--(25,9)--(30,9)--(30,6)--(34,6)--(34,2)--(33,2)--(33,5)--(21,5)--(21,21)--(5,21)--(5,33)--(2,33)--(2,34);
    \draw [fill=red, opacity=0.5] (15,21)--(21,21)--(21,15)--(15,15)--(15,21);
    
    \draw[line width=.45mm] (6,30)--(6,22)--(22,22)--(22,6)--(30,6);
    
    \node at (0.5,35.2){*};
    \node at (1.5,35.2){*};
    \node at (0.5,34.2){*};
    \node at (1.5,34.2){*};
    \node at (0.5,33.2){*};
    \node at (1.5,33.2){*};
    \node at (2.5,33.2){*};
    \node at (3.5,33.2){*};
    \node at (4.5,33.2){*};
    \node at (5.5,33.2){*};
    \node at (0.5,32.2){*};
    \node at (1.5,32.2){*};
    \node at (2.5,32.2){*};
    \node at (3.5,32.2){*};
    \node at (4.5,32.2){*};
    \node at (5.5,32.2){*};
    \node at (0.5,31.2){*};
    \node at (1.5,31.2){*};
    \node at (2.5,31.2){*};
    \node at (3.5,31.2){*};
    \node at (4.5,31.2){*};
    \node at (5.5,31.2){*};
    \node at (5.5,30.2){*};
    \node at (5.5,29.2){*};
    \node at (6.5,29.2){*};
    \node at (7.5,29.2){*};
    \node at (8.5,29.2){*};
    \node at (5.5,28.2){*};
    \node at (6.5,28.2){*};
    \node at (7.5,28.2){*};
    \node at (8.5,28.2){*};
    \node at (5.5,27.2){*};
    \node at (6.5,27.2){*};
    \node at (7.5,27.2){*};
    \node at (8.5,27.2){*};
    \node at (5.5,26.2){*};
    \node at (6.5,26.2){*};
    \node at (7.5,26.2){*};
    \node at (8.5,26.2){*};
    \node at (9.5,26.2){*};
    \node at (5.5,25.2){*};
    \node at (6.5,25.2){*};
    \node at (7.5,25.2){*};
    \node at (8.5,25.2){*};
    \node at (9.5,25.2){*};
    \node at (10.5,25.2){*};
    \node at (5.5,24.2){*};
    \node at (6.5,24.2){*};
    \node at (7.5,24.2){*};
    \node at (8.5,24.2){*};
    \node at (9.5,24.2){*};
    \node at (10.5,24.2){*};
    \node at (11.5,24.2){*};
    \node at (12.5,24.2){*};
    \node at (13.5,24.2){*};
    \node at (14.5,24.2){*};
    \node at (15.5,24.2){*};
    \node at (16.5,24.2){*};
    \node at (17.5,24.2){*};
    \node at (18.5,24.2){*};
    \node at (19.5,24.2){*};
    \node at (20.5,24.2){*};
    \node at (21.5,24.2){*};
    \node at (22.5,24.2){*};
    \node at (23.5,24.2){*};
    \node at (24.5,24.2){*};
    \node at (5.5,23.2){*};
    \node at (6.5,23.2){*};
    \node at (7.5,23.2){*};
    \node at (8.5,23.2){*};
    \node at (9.5,23.2){*};
    \node at (10.5,23.2){*};
    \node at (11.5,23.2){*};
    \node at (12.5,23.2){*};
    \node at (13.5,23.2){*};
    \node at (14.5,23.2){*};
    \node at (15.5,23.2){*};
    \node at (16.5,23.2){*};
    \node at (17.5,23.2){*};
    \node at (18.5,23.2){*};
    \node at (19.5,23.2){*};
    \node at (20.5,23.2){*};
    \node at (21.5,23.2){*};
    \node at (22.5,23.2){*};
    \node at (23.5,23.2){*};
    \node at (24.5,23.2){*};
    \node at (5.5,22.2){*};
    \node at (6.5,22.2){*};
    \node at (7.5,22.2){*};
    \node at (8.5,22.2){*};
    \node at (9.5,22.2){*};
    \node at (10.5,22.2){*};
    \node at (11.5,22.2){*};
    \node at (12.5,22.2){*};
    \node at (13.5,22.2){*};
    \node at (14.5,22.2){*};
    \node at (15.5,22.2){*};
    \node at (16.5,22.2){*};
    \node at (17.5,22.2){*};
    \node at (18.5,22.2){*};
    \node at (19.5,22.2){*};
    \node at (20.5,22.2){*};
    \node at (21.5,22.2){*};
    \node at (22.5,22.2){*};
    \node at (23.5,22.2){*};
    \node at (24.5,22.2){*};
    \node at (5.5,21.2){*};
    \node at (6.5,21.2){*};
    \node at (7.5,21.2){*};
    \node at (8.5,21.2){*};
    \node at (9.5,21.2){*};
    \node at (10.5,21.2){*};
    \node at (11.5,21.2){*};
    \node at (12.5,21.2){*};
    \node at (13.5,21.2){*};
    \node at (14.5,21.2){*};
    \node at (15.5,21.2){*};
    \node at (16.5,21.2){*};
    \node at (17.5,21.2){*};
    \node at (18.5,21.2){*};
    \node at (19.5,21.2){*};
    \node at (20.5,21.2){*};
    \node at (21.5,21.2){*};
    \node at (22.5,21.2){*};
    \node at (23.5,21.2){*};
    \node at (24.5,21.2){*};
    \node at (15.5,20.2){*};
    \node at (16.5,20.2){*};
    \node at (17.5,20.2){*};
    \node at (18.5,20.2){*};
    \node at (19.5,20.2){*};
    \node at (20.5,20.2){*};
    \node at (21.5,20.2){*};
    \node at (22.5,20.2){*};
    \node at (23.5,20.2){*};
    \node at (24.5,20.2){*};
    \node at (15.5,19.2){*};
    \node at (16.5,19.2){*};
    \node at (17.5,19.2){*};
    \node at (18.5,19.2){*};
    \node at (19.5,19.2){*};
    \node at (20.5,19.2){*};
    \node at (21.5,19.2){*};
    \node at (22.5,19.2){*};
    \node at (23.5,19.2){*};
    \node at (24.5,19.2){*};
    \node at (15.5,18.2){*};
    \node at (16.5,18.2){*};
    \node at (17.5,18.2){*};
    \node at (18.5,18.2){*};
    \node at (19.5,18.2){*};
    \node at (20.5,18.2){*};
    \node at (21.5,18.2){*};
    \node at (22.5,18.2){*};
    \node at (23.5,18.2){*};
    \node at (24.5,18.2){*};
    \node at (15.5,17.2){*};
    \node at (16.5,17.2){*};
    \node at (17.5,17.2){*};
    \node at (18.5,17.2){*};
    \node at (19.5,17.2){*};
    \node at (20.5,17.2){*};
    \node at (21.5,17.2){*};
    \node at (22.5,17.2){*};
    \node at (23.5,17.2){*};
    \node at (24.5,17.2){*};
    \node at (15.5,16.2){*};
    \node at (16.5,16.2){*};
    \node at (17.5,16.2){*};
    \node at (18.5,16.2){*};
    \node at (19.5,16.2){*};
    \node at (20.5,16.2){*};
    \node at (21.5,16.2){*};
    \node at (22.5,16.2){*};
    \node at (23.5,16.2){*};
    \node at (24.5,16.2){*};
    \node at (15.5,15.2){*};
    \node at (16.5,15.2){*};
    \node at (17.5,15.2){*};
    \node at (18.5,15.2){*};
    \node at (19.5,15.2){*};
    \node at (20.5,15.2){*};
    \node at (21.5,15.2){*};
    \node at (22.5,15.2){*};
    \node at (23.5,15.2){*};
    \node at (24.5,15.2){*};
    \node at (21.5,14.2){*};
    \node at (22.5,14.2){*};
    \node at (23.5,14.2){*};
    \node at (24.5,14.2){*};
    \node at (21.5,13.2){*};
    \node at (22.5,13.2){*};
    \node at (23.5,13.2){*};
    \node at (24.5,13.2){*};
    \node at (21.5,12.2){*};
    \node at (22.5,12.2){*};
    \node at (23.5,12.2){*};
    \node at (24.5,12.2){*};
    \node at (21.5,11.2){*};
    \node at (22.5,11.2){*};
    \node at (23.5,11.2){*};
    \node at (24.5,11.2){*};
    \node at (21.5,10.2){*};
    \node at (22.5,10.2){*};
    \node at (23.5,10.2){*};
    \node at (24.5,10.2){*};
    \node at (25.5,10.2){*};
    \node at (21.5,9.2){*};
    \node at (22.5,9.2){*};
    \node at (23.5,9.2){*};
    \node at (24.5,9.2){*};
    \node at (25.5,9.2){*};
    \node at (26.5,9.2){*};
    \node at (21.5,8.2){*};
    \node at (22.5,8.2){*};
    \node at (23.5,8.2){*};
    \node at (24.5,8.2){*};
    \node at (25.5,8.2){*};
    \node at (26.5,8.2){*};
    \node at (27.5,8.2){*};
    \node at (28.5,8.2){*};
    \node at (29.5,8.2){*};
    \node at (21.5,7.2){*};
    \node at (22.5,7.2){*};
    \node at (23.5,7.2){*};
    \node at (24.5,7.2){*};
    \node at (25.5,7.2){*};
    \node at (26.5,7.2){*};
    \node at (27.5,7.2){*};
    \node at (28.5,7.2){*};
    \node at (29.5,7.2){*};
    \node at (21.5,6.2){*};
    \node at (22.5,6.2){*};
    \node at (23.5,6.2){*};
    \node at (24.5,6.2){*};
    \node at (25.5,6.2){*};
    \node at (26.5,6.2){*};
    \node at (27.5,6.2){*};
    \node at (28.5,6.2){*};
    \node at (29.5,6.2){*};
    \node at (21.5,5.2){*};
    \node at (22.5,5.2){*};
    \node at (23.5,5.2){*};
    \node at (24.5,5.2){*};
    \node at (25.5,5.2){*};
    \node at (26.5,5.2){*};
    \node at (27.5,5.2){*};
    \node at (28.5,5.2){*};
    \node at (29.5,5.2){*};
    \node at (30.5,5.2){*};
    \node at (31.5,5.2){*};
    \node at (32.5,5.2){*};
    \node at (33.5,5.2){*};
    \node at (31.5,4.2){*};
    \node at (32.5,4.2){*};
    \node at (33.5,4.2){*};
    \node at (31.5,3.2){*};
    \node at (32.5,3.2){*};
    \node at (33.5,3.2){*};
    \node at (31.5,2.2){*};
    \node at (32.5,2.2){*};
    \node at (33.5,2.2){*};
    \node at (31.5,1.2){*};
    \node at (32.5,1.2){*};
    \node at (33.5,1.2){*};
    \node at (34.5,1.2){*};
    \node at (35.5,1.2){*};
    \node at (31.5,.2){*};
    \node at (32.5,.2){*};
    \node at (33.5,.2){*};
    \node at (34.5,.2){*};
    \node at (35.5,.2){*};
\end{tikzpicture}$$
    \caption{Vector space decomposition of $\g$}
    \label{fig:vspacedecompC}
\end{figure}
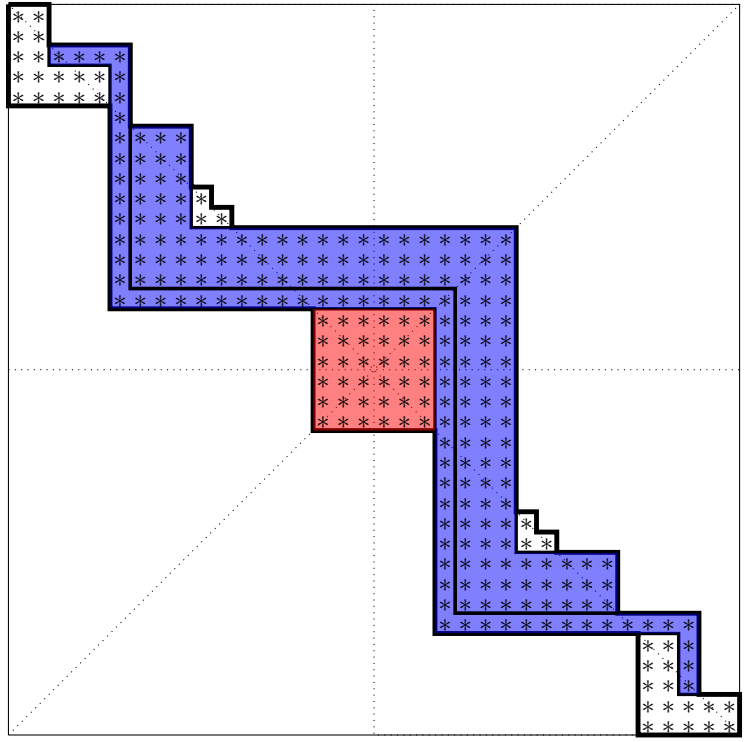

%The \textit{core} $\textgoth{C}_{\g}$ and \textit{peak} $\textgoth{P}_{\g}$ are defined the same way in type C as they are in type A, and recall that the elements of $\textgoth{C}_{\g}$ and $\textgoth{P}_{\g}$ are called \textit{core blocks} and \textit{peak blocks}. See Figure~\ref{fig:corepeakC}, where we outline the core blocks and bold the peak blocks within our running example seaweed $\g$.

There are five terms in the vector space decomposition of $\g,$ each having its own core and peak sets. In order to explicitly describe the set products in each of the core and peak sets, we implicitly assign a labeling $v_1$ through $v_{36}$ -- increasing from left to right -- to the full meander in Figure~\ref{fig:fullmeander1}, which induces a corresponding labeling of $CM(\g)$. Then the core set of the component $\mathfrak{g}|_{\bn{3}}$ is the set of $3\times 3$ blocks along the diagonal of $\mathfrak{g}|_{\bn{3}};$ specifically, 

\vspace{-1.5em}

\begin{center}
\resizebox{\linewidth}{!}{$\textgoth{C}_{\bn{3}}=\Big\{\{7,8,9\}\times\{7,8,9\},\{12,13,14\}\times\{12,13,14\},\{23,24,25\}\times\{23,24,25\},\{28,29,30\}\times\{28,29,30\}\Big\}.$}
\end{center}
The peak set of the component $\mathfrak{g}|_{\bn{3}}$ is the set of $3\times 3$ blocks $$\textgoth{P}_{\bn{3}}=\Big\{\{12,13,14\}\times\{7,8,9\},\{12,13,14\}\times\{23,24,25\},\{28,29,30\}\times\{23,24,25\}\Big\}.$$ The elements of the core and peak of $\g$ are bolded and outlined, respectively, in Figure~\ref{fig:corepeakC}.

\begin{figure}[H]
$$\begin{tikzpicture}[scale=0.27]
    \def\Node{\node [circle, fill, inner sep=1.1pt]}
    \draw (0,0)--(36,0)--(36,36)--(0,36)--(0,0);
    \draw [line width=.65mm] (0,36)--(0,31)--(5,31)--(5,21)--(15,21)--(15,15)--(21,15)--(21,5)--(31,5)--(31,0)--(36,0)--(36,2)--(34,2)--(34,6)--(30,6)--(30,9)--(27,9)--(27,10)--(26,10)--(26,11)--(25,11)--(25,25)--(11,25)--(11,26)--(10,26)--(10,27)--(9,27)--(9,30)--(6,30)--(6,34)--(2,34)--(2,36)--(0,36);
    \draw [dotted] (0,0)--(36,36)--(0,36)--(36,0)--(18,0)--(18,36);
    \draw [dotted] (0,18)--(36,18);

    \draw [fill=blue, fill opacity=0.5] (2,34)--(6,34)--(6,30)--(9,30)--(9,25)--(25,25)--(25,9)--(30,9)--(30,6)--(34,6)--(34,2)--(33,2)--(33,5)--(21,5)--(21,21)--(5,21)--(5,33)--(2,33)--(2,34);
    \draw [fill=red, opacity=0.5] (15,21)--(21,21)--(21,15)--(15,15)--(15,21);
    
    \draw[line width=.45mm] (6,30)--(6,22)--(22,22)--(22,6)--(30,6);
    
    \draw (0,33)--(2,33)--(2,31);
    \draw (6,33)--(5,33)--(5,34);
    \draw (9,26)--(10,26)--(10,25);
    \draw (6,21)--(6,22)--(5,22);
    \draw (6,25)--(9,25)--(9,22);
    \draw (22,25)--(22,22)--(25,22);
    \draw (21,22)--(21,21)--(22,21);
    \draw (25,6)--(25,9)--(22,9);
    \draw (21,6)--(22,6)--(22,5);
    \draw (26,9)--(26,10)--(25,10);
    \draw (33,6)--(33,5)--(34,5);
    \draw (33,0)--(33,2)--(31,2);

    \node at (0.5,35.2){\large\bf*};
    \node at (1.5,35.2){\large\bf*};
    \node at (0.5,34.2){\large\bf*};
    \node at (1.5,34.2){\large\bf*};
    \node at (0.5,33.2){*};
    \node at (1.5,33.2){*};
    \node at (2.5,33.2){\large\bf*};
    \node at (3.5,33.2){*};
    \node at (4.5,33.2){*};
    \node at (5.5,33.2){*};
    \node at (0.5,32.2){*};
    \node at (1.5,32.2){*};
    \node at (2.5,32.2){*};
    \node at (3.5,32.2){\large\bf*};
    \node at (4.5,32.2){\large\bf*};
    \node at (5.5,32.2){*};
    \node at (0.5,31.2){*};
    \node at (1.5,31.2){*};
    \node at (2.5,31.2){*};
    \node at (3.5,31.2){\large\bf*};
    \node at (4.5,31.2){\large\bf*};
    \node at (5.5,31.2){*};
    \node at (5.5,30.2){\large\bf*};
    \node at (5.5,29.2){*};
    \node at (6.5,29.2){\large\bf*};
    \node at (7.5,29.2){\large\bf*};
    \node at (8.5,29.2){\large\bf*};
    \node at (5.5,28.2){*};
    \node at (6.5,28.2){\large\bf*};
    \node at (7.5,28.2){\large\bf*};
    \node at (8.5,28.2){\large\bf*};
    \node at (5.5,27.2){*};
    \node at (6.5,27.2){\large\bf*};
    \node at (7.5,27.2){\large\bf*};
    \node at (8.5,27.2){\large\bf*};
    \node at (5.5,26.2){*};
    \node at (6.5,26.2){*};
    \node at (7.5,26.2){*};
    \node at (8.5,26.2){*};
    \node at (9.5,26.2){\large\bf*};
    \node at (5.5,25.2){*};
    \node at (6.5,25.2){*};
    \node at (7.5,25.2){*};
    \node at (8.5,25.2){*};
    \node at (9.5,25.2){*};
    \node at (10.5,25.2){\large\bf*};
    \node at (5.5,24.2){*};
    \node at (6.5,24.2){*};
    \node at (7.5,24.2){*};
    \node at (8.5,24.2){*};
    \node at (9.5,24.2){*};
    \node at (10.5,24.2){*};
    \node at (11.5,24.2){\large\bf*};
    \node at (12.5,24.2){\large\bf*};
    \node at (13.5,24.2){\large\bf*};
    \node at (14.5,24.2){*};
    \node at (15.5,24.2){*};
    \node at (16.5,24.2){*};
    \node at (17.5,24.2){*};
    \node at (18.5,24.2){*};
    \node at (19.5,24.2){*};
    \node at (20.5,24.2){*};
    \node at (21.5,24.2){*};
    \node at (22.5,24.2){*};
    \node at (23.5,24.2){*};
    \node at (24.5,24.2){*};
    \node at (5.5,23.2){*};
    \node at (6.5,23.2){*};
    \node at (7.5,23.2){*};
    \node at (8.5,23.2){*};
    \node at (9.5,23.2){*};
    \node at (10.5,23.2){*};
    \node at (11.5,23.2){\large\bf*};
    \node at (12.5,23.2){\large\bf*};
    \node at (13.5,23.2){\large\bf*};
    \node at (14.5,23.2){*};
    \node at (15.5,23.2){*};
    \node at (16.5,23.2){*};
    \node at (17.5,23.2){*};
    \node at (18.5,23.2){*};
    \node at (19.5,23.2){*};
    \node at (20.5,23.2){*};
    \node at (21.5,23.2){*};
    \node at (22.5,23.2){*};
    \node at (23.5,23.2){*};
    \node at (24.5,23.2){*};
    \node at (5.5,22.2){*};
    \node at (6.5,22.2){*};
    \node at (7.5,22.2){*};
    \node at (8.5,22.2){*};
    \node at (9.5,22.2){*};
    \node at (10.5,22.2){*};
    \node at (11.5,22.2){\large\bf*};
    \node at (12.5,22.2){\large\bf*};
    \node at (13.5,22.2){\large\bf*};
    \node at (14.5,22.2){*};
    \node at (15.5,22.2){*};
    \node at (16.5,22.2){*};
    \node at (17.5,22.2){*};
    \node at (18.5,22.2){*};
    \node at (19.5,22.2){*};
    \node at (20.5,22.2){*};
    \node at (21.5,22.2){*};
    \node at (22.5,22.2){*};
    \node at (23.5,22.2){*};
    \node at (24.5,22.2){*};
    \node at (5.5,21.2){*};
    \node at (6.5,21.2){*};
    \node at (7.5,21.2){*};
    \node at (8.5,21.2){*};
    \node at (9.5,21.2){*};
    \node at (10.5,21.2){*};
    \node at (11.5,21.2){*};
    \node at (12.5,21.2){*};
    \node at (13.5,21.2){*};
    \node at (14.5,21.2){\large\bf*};
    \node at (15.5,21.2){*};
    \node at (16.5,21.2){*};
    \node at (17.5,21.2){*};
    \node at (18.5,21.2){*};
    \node at (19.5,21.2){*};
    \node at (20.5,21.2){*};
    \node at (21.5,21.2){*};
    \node at (22.5,21.2){*};
    \node at (23.5,21.2){*};
    \node at (24.5,21.2){*};
    \node at (15.5,20.2){\large\bf*};
    \node at (16.5,20.2){\large\bf*};
    \node at (17.5,20.2){\large\bf*};
    \node at (18.5,20.2){\large\bf*};
    \node at (19.5,20.2){\large\bf*};
    \node at (20.5,20.2){\large\bf*};
    \node at (21.5,20.2){*};
    \node at (22.5,20.2){*};
    \node at (23.5,20.2){*};
    \node at (24.5,20.2){*};
    \node at (15.5,19.2){\large\bf*};
    \node at (16.5,19.2){\large\bf*};
    \node at (17.5,19.2){\large\bf*};
    \node at (18.5,19.2){\large\bf*};
    \node at (19.5,19.2){\large\bf*};
    \node at (20.5,19.2){\large\bf*};
    \node at (21.5,19.2){*};
    \node at (22.5,19.2){*};
    \node at (23.5,19.2){*};
    \node at (24.5,19.2){*};
    \node at (15.5,18.2){\large\bf*};
    \node at (16.5,18.2){\large\bf*};
    \node at (17.5,18.2){\large\bf*};
    \node at (18.5,18.2){\large\bf*};
    \node at (19.5,18.2){\large\bf*};
    \node at (20.5,18.2){\large\bf*};
    \node at (21.5,18.2){*};
    \node at (22.5,18.2){*};
    \node at (23.5,18.2){*};
    \node at (24.5,18.2){*};
    \node at (15.5,17.2){\large\bf*};
    \node at (16.5,17.2){\large\bf*};
    \node at (17.5,17.2){\large\bf*};
    \node at (18.5,17.2){\large\bf*};
    \node at (19.5,17.2){\large\bf*};
    \node at (20.5,17.2){\large\bf*};
    \node at (21.5,17.2){*};
    \node at (22.5,17.2){*};
    \node at (23.5,17.2){*};
    \node at (24.5,17.2){*};
    \node at (15.5,16.2){\large\bf*};
    \node at (16.5,16.2){\large\bf*};
    \node at (17.5,16.2){\large\bf*};
    \node at (18.5,16.2){\large\bf*};
    \node at (19.5,16.2){\large\bf*};
    \node at (20.5,16.2){\large\bf*};
    \node at (21.5,16.2){*};
    \node at (22.5,16.2){*};
    \node at (23.5,16.2){*};
    \node at (24.5,16.2){*};
    \node at (15.5,15.2){\large\bf*};
    \node at (16.5,15.2){\large\bf*};
    \node at (17.5,15.2){\large\bf*};
    \node at (18.5,15.2){\large\bf*};
    \node at (19.5,15.2){\large\bf*};
    \node at (20.5,15.2){\large\bf*};
    \node at (21.5,15.2){*};
    \node at (22.5,15.2){*};
    \node at (23.5,15.2){*};
    \node at (24.5,15.2){*};
    \node at (21.5,14.2){\large\bf*};
    \node at (22.5,14.2){*};
    \node at (23.5,14.2){*};
    \node at (24.5,14.2){*};
    \node at (21.5,13.2){*};
    \node at (22.5,13.2){\large\bf*};
    \node at (23.5,13.2){\large\bf*};
    \node at (24.5,13.2){\large\bf*};
    \node at (21.5,12.2){*};
    \node at (22.5,12.2){\large\bf*};
    \node at (23.5,12.2){\large\bf*};
    \node at (24.5,12.2){\large\bf*};
    \node at (21.5,11.2){*};
    \node at (22.5,11.2){\large\bf*};
    \node at (23.5,11.2){\large\bf*};
    \node at (24.5,11.2){\large\bf*};
    \node at (21.5,10.2){*};
    \node at (22.5,10.2){*};
    \node at (23.5,10.2){*};
    \node at (24.5,10.2){*};
    \node at (25.5,10.2){\large\bf*};
    \node at (21.5,9.2){*};
    \node at (22.5,9.2){*};
    \node at (23.5,9.2){*};
    \node at (24.5,9.2){*};
    \node at (25.5,9.2){*};
    \node at (26.5,9.2){\large\bf*};
    \node at (21.5,8.2){*};
    \node at (22.5,8.2){*};
    \node at (23.5,8.2){*};
    \node at (24.5,8.2){*};
    \node at (25.5,8.2){*};
    \node at (26.5,8.2){*};
    \node at (27.5,8.2){\large\bf*};
    \node at (28.5,8.2){\large\bf*};
    \node at (29.5,8.2){\large\bf*};
    \node at (21.5,7.2){*};
    \node at (22.5,7.2){*};
    \node at (23.5,7.2){*};
    \node at (24.5,7.2){*};
    \node at (25.5,7.2){*};
    \node at (26.5,7.2){*};
    \node at (27.5,7.2){\large\bf*};
    \node at (28.5,7.2){\large\bf*};
    \node at (29.5,7.2){\large\bf*};
    \node at (21.5,6.2){*};
    \node at (22.5,6.2){*};
    \node at (23.5,6.2){*};
    \node at (24.5,6.2){*};
    \node at (25.5,6.2){*};
    \node at (26.5,6.2){*};
    \node at (27.5,6.2){\large\bf*};
    \node at (28.5,6.2){\large\bf*};
    \node at (29.5,6.2){\large\bf*};
    \node at (21.5,5.2){*};
    \node at (22.5,5.2){*};
    \node at (23.5,5.2){*};
    \node at (24.5,5.2){*};
    \node at (25.5,5.2){*};
    \node at (26.5,5.2){*};
    \node at (27.5,5.2){*};
    \node at (28.5,5.2){*};
    \node at (29.5,5.2){*};
    \node at (30.5,5.2){\large\bf*};
    \node at (31.5,5.2){*};
    \node at (32.5,5.2){*};
    \node at (33.5,5.2){*};
    \node at (31.5,4.2){\large\bf*};
    \node at (32.5,4.2){\large\bf*};
    \node at (33.5,4.2){*};
    \node at (31.5,3.2){\large\bf*};
    \node at (32.5,3.2){\large\bf*};
    \node at (33.5,3.2){*};
    \node at (31.5,2.2){*};
    \node at (32.5,2.2){*};
    \node at (33.5,2.2){\large\bf*};
    \node at (31.5,1.2){*};
    \node at (32.5,1.2){*};
    \node at (33.5,1.2){*};
    \node at (34.5,1.2){\large\bf*};
    \node at (35.5,1.2){\large\bf*};
    \node at (31.5,.2){*};
    \node at (32.5,.2){*};
    \node at (33.5,.2){*};
    \node at (34.5,.2){\large\bf*};
    \node at (35.5,.2){\large\bf*};
\end{tikzpicture}$$
    \caption{The core and peak identified within $\g$}
    \label{fig:corepeakC}
\end{figure}
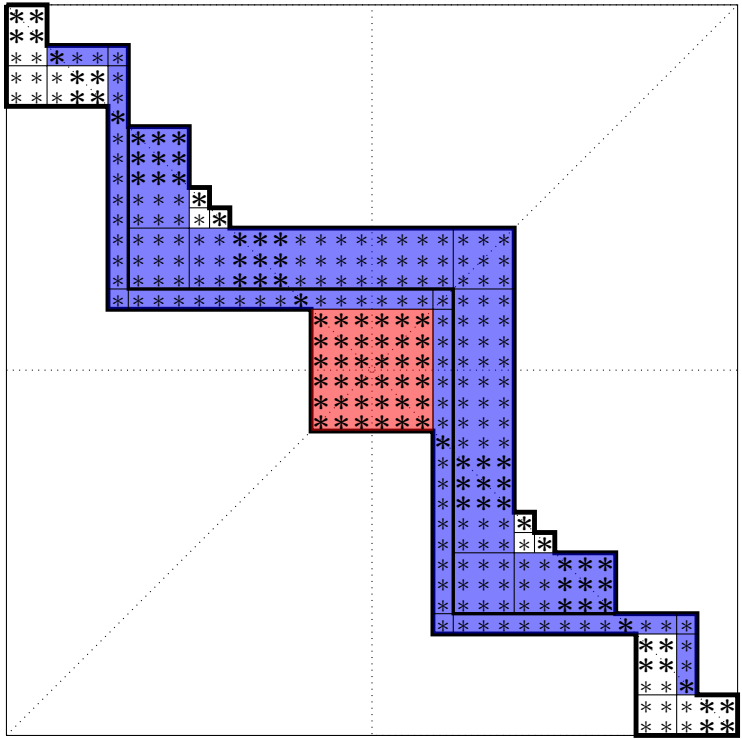

\end{Ex}

\subsection{Regular Forms}\label{sec:Cframework}
The core and peak of an index-one seaweed $\g$ assist in constructing a family $\Phi_C$ of regular one-forms on $\g$ that are dependent upon the homotopy type of $\g.$ Utilizing the symmetry conditions across the antidiagonal in a type-C seaweed, we can, without loss of generality, express an element of $\Phi_C$ as $\sum e_{i,j}^*$, where the sum is over a restricted set of admissible locations $(i,j)$ of $\g$ for which $i+j\leq 2n+1.$ The restricted set of admissible locations is visualized by placing a black dot in a location ($i,j$) if and only if $e_{i,j}^*$ is a summand of the regular one-form (see Example~\ref{ex:genCfunctional}). Since the restricted set of locations is determined by the homotopy type of $\mathfrak{g}$, we will explicitly construct a regular one-form $\varphi\in\Phi_C$ for each homotopy type listed in Corollary~\ref{cor:ind1htC}. First, we briefly describe the algorithm for defining $\varphi\in\Phi_C$ on a seaweed with arbitrary homotopy type and index.

%That is, we need only place black dots above (or on) the antidiagonal of the seaweed in order to define an element of $\Phi_C.$ 

%%%%%%%%%%%%%%%%%%%%%%%%%%%%%%%%%%%%%%%%%%%%%%%%%%%%
%%%%%%%%%%%%%%%%%%%%%%%%%%%%%%%%%%%%%%%%%%%%%%%%%%%%

%Since type-C seaweeds satisfy symmetry conditions across the antidiagonal, it is only necessary to place black dots on or above the antidiagonal when constructing a regular one-form $\varphi\in\Phi_C$. That is, if $\g=\mathfrak{p}_{2n}^C \frac{a_1|\dots|a_m}{b_1|\dots|b_t},$ the only potential locations for black dots which define $\varphi$ are $(i,j)$ for $i+j\leq 2n+1.$ With this restriction in mind, we can now fully describe $\varphi$ for a general type-C seaweed.

Begin by identifying the tail and aftertail components in the seaweed. For tail components, form a triangle of black dots of height $\lfloor \frac{c_i}{2}\rfloor,$ where $c_i$ is the size of the tail component, in the upper left-hand corner of each core block and in all entries on the diagonal of each peak block, so long as these locations $(i,j)$ satisfy $i+j\leq 2n+1.$ For the aftertail component, form a triangle of black dots of height $\frac{c_i}{2},$ where $c_i$ is the size of the aftertail component, in the upper left-hand corner of the (unique) core block. Note there are no peak blocks in an aftertail component. The only remaining components are those which are neither tail nor aftertail components. For components of this form, black dots are placed on and above the antidiagonal of each core block and on the diagonal of each peak block. Example~\ref{ex:genCfunctional} illustrates this general construction. 

\begin{Ex}\label{ex:genCfunctional}
The regular one-form $\varphi$ can be visualized in the seaweed via the placement of black dots, as seen in Figure~\ref{fig:genCfunctional}.

\begin{figure}[H]
$$\begin{tikzpicture}[scale=0.2]
    \def\Node{\node [circle, fill, inner sep=1.1pt]}
    \draw (0,0)--(36,0)--(36,36)--(0,36)--(0,0);
    \draw [line width=.65mm] (0,36)--(0,31)--(5,31)--(5,21)--(15,21)--(15,15)--(21,15)--(21,5)--(31,5)--(31,0)--(36,0)--(36,2)--(34,2)--(34,6)--(30,6)--(30,9)--(27,9)--(27,10)--(26,10)--(26,11)--(25,11)--(25,25)--(11,25)--(11,26)--(10,26)--(10,27)--(9,27)--(9,30)--(6,30)--(6,34)--(2,34)--(2,36)--(0,36);
    \draw [dotted] (0,0)--(36,36)--(0,36)--(36,0)--(18,0)--(18,36);
    \draw [dotted] (0,18)--(36,18);

    \draw [fill=blue, fill opacity=0.5] (2,34)--(6,34)--(6,30)--(9,30)--(9,25)--(25,25)--(25,9)--(30,9)--(30,6)--(34,6)--(34,2)--(33,2)--(33,5)--(21,5)--(21,21)--(5,21)--(5,33)--(2,33)--(2,34);
    \draw [fill=red, opacity=0.5] (15,21)--(21,21)--(21,15)--(15,15)--(15,21);
    
    \draw[line width=.45mm] (6,30)--(6,22)--(22,22)--(22,6)--(30,6);
    
    \draw (0,33)--(2,33)--(2,31);
    \draw (6,33)--(5,33)--(5,34);
    \draw (9,26)--(10,26)--(10,25);
    \draw (6,21)--(6,22)--(5,22);
    \draw (6,25)--(9,25)--(9,22);
    \draw (22,25)--(22,22)--(25,22);
    \draw (21,22)--(21,21)--(22,21);
    \draw (25,6)--(25,9)--(22,9);
    \draw (21,6)--(22,6)--(22,5);
    \draw (26,9)--(26,10)--(25,10);
    \draw (33,6)--(33,5)--(34,5);
    \draw (33,0)--(33,2)--(31,2);
    
    \Node at (0.5,35.5){};
    \Node at (0.5,34.5){};
    \Node at (0.5,32.5){};
    \Node at (1.5,35.5){};
    \Node at (1.5,31.5){};
    \Node at (3.5,31.5){};
    \Node at (4.5,32.5){};
    \Node at (3.5,32.5){};
    \Node at (5.5,33.5){};
    \Node at (5.5,21.5){};
    \Node at (6.5,29.5){};
    \Node at (6.5,24.5){};
    \Node at (7.5,23.5){};
    \Node at (8.5,22.5){};
    \Node at (9.5,25.5){};
    \Node at (11.5,24.5){};
    \Node at (15.5,20.5){};
    \Node at (15.5,19.5){};
    \Node at (15.5,18.5){};
    \Node at (16.5,20.5){};
    \Node at (16.5,19.5){};
    \Node at (17.5,20.5){};
    \Node at (21.5,21.5){};
    \Node at (22.5,24.5){};
    \Node at (23.5,23.5){};
\end{tikzpicture}$$
\caption{Summands of $\varphi$ given as black dots in the seaweed $\mathfrak{p}_{18}^C\frac{5|10}{2|4|3|1|1}$}
\label{fig:genCfunctional}
\end{figure}
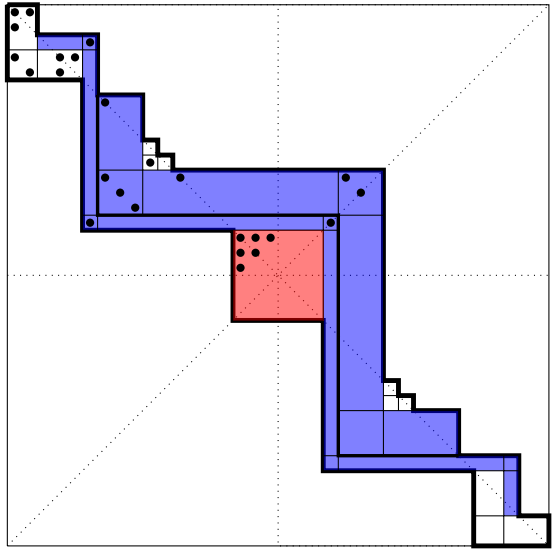
The one-form $\varphi$ is then given by \begin{align*}
\varphi=& e_{1,1}^*+e_{1,2}^*+e_{2,1}^*+e_{3,6}^*+e_{4,1}^*+e_{4,4}^*+e_{4,5}^*+e_{5,2}^*+e_{5,4}^*+e_{7,7}^*+e_{11,10}^*+e_{12,7}^*+e_{12,12}^*+e_{12,23}^* \\ &+e_{13,8}^*+e_{13,24}^*+e_{14,9}^*+e_{15,6}^*+e_{15,22}^*+e_{16,16}^*+e_{16,17}^*+e_{16,18}^*+e_{17,16}^*+e_{17,17}^*+e_{18,16}^*.
\end{align*}

\end{Ex}

The description and construction of elements of $\Phi_C$ are simplified upon restriction to index-one seaweeds. We proceed by treating each of the homotopy types in Corollary~\ref{cor:ind1htC} as its own case.

%By Remark~\ref{rem:Chomtype}, there are only four distinct homotopy types an index-one, type-C seaweed can have. We proceed by constructing $\varphi$ in each case, involving the meander whenever possible. First, we define some notation. Let $E\left(\overrightarrow{M}_n^C\right)$ be the edge set of $\overrightarrow{M}_n^C,$ $E\left(\overrightarrow{M}(\g)\right)$ be the edge set of $\overrightarrow{M}(\g),$ and let $E_T\left(\overrightarrow{M}(\g)\right)\subset E\left(\overrightarrow{M}(\g)\right)$ be the set of directed edges which have one endpoint in the tail.

\subsubsection{$\mathcal{H}_C({\color{blue}1,\dots,1,c,1,\dots,1})$ for $\bn{c}=\bn{2}$ or $\bn{3}$}\label{sec:h3}
Let $\g=\mathfrak{p}_{2n}^C\frac{a_1|\dots|a_m}{b_1|\dots|b_t}$ have homotopy type $\mathcal{H}_C({\color{blue}1,\dots,1,c,1,\dots,1}),$ where $\bn{c}=\bn{2}$ or $\bn{3}.$ The core blocks and peak blocks differ in dimension depending on the value of $\bn{c};$ however, the construction and proofs here considered hold whether $\bn{c}=\bn{2}$ or $\bn{3}.$ Denote by $\varphi_{(\bn{c})}$ the regular one-form on $\g$ constructed as in Section~\ref{sec:Cframework}. See Example~\ref{ex:h2cfunctional}.

\begin{Ex}\label{ex:h2cfunctional}
Consider the seaweed $\g=\mathfrak{p}_{14}^C\frac{7}{1|2},$ which has homotopy type $\mathcal{H}_C({\color{blue}1,2,1}).$ The regular one-form $\varphi_{(\bn{2})}$ has summands determined by the locations of black dots in Figure~\ref{fig:h2C}.
\begin{figure}[H]
$$\begin{tikzpicture}[scale=0.5]
    \def\Node{\node [circle, fill, inner sep=1.5pt]}
    \draw (0,0)--(14,0)--(14,14)--(0,14)--(0,0);
    \draw [line width=.65mm, fill=blue, opacity=0.5] (0,14)--(0,7)--(7,7)--(7,0)--(14,0)--(14,1)--(13,1)--(13,3)--(11,3)--(11,11)--(3,11)--(3,13)--(1,13)--(1,14)--(0,14);
    \draw [dotted] (0,0)--(14,14)--(0,14)--(14,0)--(7,0)--(7,14);
    \draw [dotted] (0,7)--(14,7);
    
    \draw (0,13)--(1,13)--(1,8)--(8,8)--(8,1)--(13,1)--(13,0);
    \draw (6,7)--(6,8)--(7,8)--(7,7)--(8,7)--(8,6)--(7,6);
    
    \draw (3,11)--(3,10)--(10,10)--(10,3)--(11,3);
    \draw (4,11)--(4,10);
    \draw (11,4)--(10,4);
    
    \draw (1,11)--(3,11);
    \draw (4,10)--(4,8);
    \draw (6,10)--(6,8);
    \draw (8,6)--(10,6);
    \draw (8,4)--(10,4);
    \draw (11,3)--(11,1);
    
    \Node at (0.5,7.5){};
    \Node at (7.5,7.5){};
    \Node at (1.5,12.5){};
    \Node at (1.5,9.5){};
    \Node at (2.5,8.5){};
    \Node at (4.5,9.5){};
    \Node at (8.5,9.5){};
    \Node at (10.5,10.5){};
\end{tikzpicture}$$
    \caption{The summands of $\varphi_{(\bn{2})}$ identified within $\g$}
    \label{fig:h2C}
\end{figure}

\noindent
As Figure~\ref{fig:h2C} displays, $\varphi_{(\bn{2})}$ is given by $$\varphi_{(\bn{2})}=e_{2,2}^*+e_{4,11}^*+e_{5,2}^*+e_{5,5}^*+e_{5,9}^*+e_{6,3}^*+e_{7,1}^*+e_{7,8}^*.$$
\end{Ex}

For seaweeds with homotopy type $\mathcal{H}_C(\bn{1,\dots,1,c,1,\dots,1}),$ where $\bn{c}=\bn{2}$ or $\bn{3},$ the full meander $M_{2n}^C(\mathfrak{g})$ consists of paths, each with one vertex in the tail, and exactly one cycle. Such a cycle, along with the following technical lemma, will allow for the identification of the generator of $\ker(B_{\varphi_{(\bn{c})}})$ in Theorem~\ref{thm:h2C} below.

\begin{lemma}\label{lem:oppsignC}
Let $\g\subset\mathfrak{sp}(2n)$ be a type-C seaweed with homotopy type consisting of at least one component of size $\bn{c}=\bn{2}$ or $\bn{3}$. Then let $\mathscr{C}=(v_{i_1},v_{i_2},\dots,v_{i_k},v_{i_1})$ be a cycle in the full meander $M_{2n}^C(\g)$ corresponding to a component of size $\bn{c}$, and let $h_{(\bn{c})}=\sum_{j=1}^{k}(-1)^{j+1}e_{i_j,i_j}.$ If $v_{i_r},v_{i_{\ell}}\in \mathscr{C}$ and $r,\ell\in I,$ where $v_I$ is a vertex in the component meander $CM(\mathfrak{g}),$ then $e_{i_r,i_r}$ and $e_{i_{\ell},i_{\ell}}$ have coefficients with opposite sign in $h_{(\bn{c})}.$
\end{lemma}
\begin{proof}
Since $\g$ has a component of size $\bn{c}$, there exist vertices $v_i$ and $v_{i+\bn{c}-1}$ in $M_{2n}^C(\g)$ such that $\{v_i,v_{i+\bn{c}-1}\}$ is an edge in $M_{2n}^C(\g).$ Fix such vertices $v_i$ and $v_{i+\bn{c}-1}.$ Considering $M_{2n}^C(\g),$ since $v_{i_r}$ and $v_{i_{\ell}}$ are identified in the construction of $CM(\mathfrak{g})$ by assumption, it follows that the length of the shortest path from $v_{i_r}$ to $v_i$ is equal to the length of the shortest path from $v_{i_{\ell}}$ to $v_{i+\bn{c}-1}.$ Therefore, the length of the path from $v_{i_r}$ to $v_{i_{\ell}}$ is odd, so $e_{i_r,i_r}$ and $e_{i_{\ell},i_{\ell}}$ have coefficients with opposite sign in $h_{(\bn{c})}.$
\end{proof}

\begin{theorem}\label{thm:h2C}
Let $\mathfrak{g}\subset\mathfrak{sp}(2n)$ be a type-C seaweed with homotopy type $\mathcal{H}_C(\bn{1,\dots,1,c,1,\dots,1}),$ where $\bn{c}=\bn{2}$ or $\bn{3},$ and let $\mathscr{C}=(v_{i_1},v_{i_2},\dots,v_{i_k},v_{i_1})$ be the unique cycle in the full meander $M_{2n}^C(\g).$ If $\varphi_{(\bn{c})}$ is defined as in Section~\ref{sec:Cframework}, then $$\ker(B_{\varphi_{(\bn{c})}})=\text{span}\{h_{(\bn{c})}\},$$ where $\bn{c}=\bn{2}$ or $\bn{3},$ and $$h_{(\bn{c})}=\sum_{j=1}^k(-1)^{j+1}e_{i_j,i_j}.$$
\end{theorem}
\begin{proof}
Since $\varphi_{(\bn{c})}$ is regular, we need only show that $h_{(\bn{c})}\in\ker(B_{\varphi_{(\bn{c})}}),$ for $\bn{c}=\bn{2}$ or $\bn{3}.$ Note that by the definition of the full meander $M_{2n}^C(\g),$ the length of any path from $v_{i_j}$ to $v_{2n-i_j+1}$ is odd, so the coefficients of $e_{i_j,i_j}$ and $e_{2n-i_j+1,2n-i_j+1}$ in the expansion of $h_{(\bn{c})}$ have opposite sign, so $h_{(\bn{c})}\in\g$ for $\bn{c}=\bn{2}$ or $\bn{3}.$ Now, to see that $h_{(\bn{c})}\in\ker(B_{\varphi_{(\bn{c})}}),$ consider the following:
\begin{enumerate}
    \item[\textbf{(a)}] $\varphi_{(\bn{c})}([h_{(\bn{c})},e_{i,i}-e_{2n-i+1,2n-i+1}])=0,$ for all $1\leq i\leq n,$
    \item[\textbf{(b)}] $\varphi_{(\bn{c})}([h_{(\bn{c})},e_{i,j}\pm e_{2n-j+1,2n-i+1}])=0,$ for all $i,j$ such that $i+j<2n+1$ and $e_{i,j}^*$ is not a summand of $\varphi_{(\bn{c})},$
    \item[\textbf{(c)}] $\varphi_{(\bn{c})}([h_{(\bn{c})},e_{i,j}])=0,$ for all $i,j$ such that $i+j=2n+1$ and $e_{i,j}^*$ is not a summand of $\varphi_{(\bn{c})},$
    \item[\textbf{(d)}] $\varphi_{(\bn{c})}([h_{(\bn{c})},e_{i,j}\pm e_{2n-j+1,2n-i+1}])=\varphi_{(\bn{c})}\left(e_{i,j}\pm e_{2n-j+1,2n-i+1}-(e_{i,j}\pm e_{2n-j+1,2n-i+1})\right)=0,$ for all $i,j$ such that $i+j<2n+1$ and $e_{i,j}^*$ is a summand of $\varphi_{(\bn{c})},$ and
    \item[\textbf{(e)}] $\varphi_{(\bn{c})}([h_{(\bn{c})},e_{i,j}])=\varphi_{(\bn{c})}(e_{i,j}-e_{i,j})=0,$ for all $i,j$ such that $i+j=2n+1$ and $e_{i,j}^*$ is a summand of $\varphi_{(\bn{c})}.$
\end{enumerate}
Equations~\textbf{(a)-(c)} follow immediately from $h_{(\bn{c})}$ being a semisimple element of $\g,$ and Equation~\textbf{(d)} follows from Lemma~\ref{lem:oppsignC} in the following way. 

Without loss of generality, let $\bn{c}=\bn{2}.$ In $\g|_{\bn{2}},$ each peak block consists of exactly four entries. Fix such a peak block and call the entries -- from left to right, top to bottom -- $(s,t),(s,t+1),(s+1,t),$ and $(s+1,t+1).$ The entries on the antidiagonal of the peak block, $(s,t+1)$ and $(s+1,t),$ correspond to the edges $(v_s,v_{t+1})$ and $(v_{s+1},v_t)$ of the oriented full meander $\overrightarrow{M}_{2n}^C(\g).$ Therefore, $e_{s,s}$ and $e_{t+1,t+1}$ have coefficients with opposite sign in the expansion of $h_{(\bn{2})},$ and similarly for the summands $e_{s+1,s+1}$ and $e_{t,t}.$ Further, the vertices $v_t$ and $v_{t+1}$ of $M_{2n}^C(\g)$ are identified in the construction of the component meander $CM(\mathfrak{g}),$ so Lemma~\ref{lem:oppsignC} implies that $e_{t,t}$ and $e_{t+1,t+1}$ have coefficients with opposite sign in the expansion of $h_{(\bn{2})}.$ Therefore, the coefficients of $e_{s,s}$ and $e_{t,t}$ in the expansion of $h_{(\bn{2})}$ have the same sign, and similarly for $e_{s+1,s+1}$ and $e_{t+1,t+1}.$ We then conclude that Equation \textbf{(d)} holds since $e_{s,t}^*$ and $e_{s+1,t+1}^*$ are the summands of $\varphi_{(\bn{2})}$ arising from the chosen peak block. Equation \textbf{(e)} follows via similar reasoning.
\end{proof}

\begin{Ex}\label{ex:h2ckernel}
Returning to the seaweed $\mathfrak{g}=\mathfrak{p}_{14}^C\frac{7}{1|2}$ from Example~\ref{ex:h2cfunctional}, note that the cycle in the full meander $M_{2n}^C(\mathfrak{g})$ \textup(see Figure~\ref{fig:h2cmeander}\textup) can be written as $$(v_2,v_6,v_9,v_{13},v_{12},v_{10},v_5,v_3,v_2),$$ so we have that $$\ker(B_{\varphi_{(\bn{2})}})=\text{span}\{e_{2,2}-e_{6,6}+e_{9,9}-e_{13,13}+e_{12,12}-e_{10,10}+e_{5,5}-e_{3,3}\}.$$

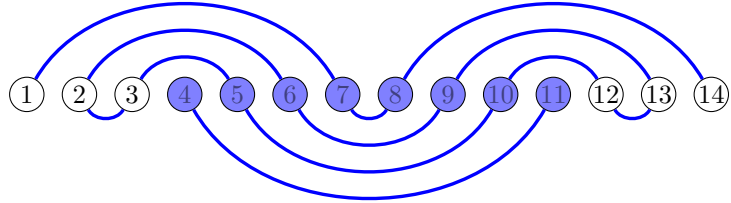
\begin{figure}[H]
$$\begin{tikzpicture}[scale=0.7]
    \def\Node{\node [circle, fill, inner sep=1.5pt]}
    \vertex (1) at (1,0){1};
    \vertex (2) at (2,0){2};
    \vertex (3) at (3,0){3};
    \vertex[fill=blue, fill opacity=0.5, text opacity=1] (4) at (4,0){4};
    \vertex[fill=blue, fill opacity=0.5, text opacity=1] (5) at (5,0){5};
    \vertex[fill=blue, fill opacity=0.5, text opacity=1] (6) at (6,0){6};
    \vertex[fill=blue, fill opacity=0.5, text opacity=1] (7) at (7,0){7};
   
    \vertex[fill=blue, fill opacity=0.5, text opacity=1] (700) at (8,0){8};
    \vertex[fill=blue, fill opacity=0.5, text opacity=1] (600) at (9,0){9};
    \vertex[fill=blue, fill opacity=0.5, text opacity=1] (500) at (10,0){10};
    \vertex[fill=blue, fill opacity=0.5, text opacity=1] (400) at (11,0){11};
    \vertex (300) at (12,0){12};
    \vertex (200) at (13,0){13};
    \vertex (100) at (14,0){14};
    \draw [line width=0.45mm, blue] (1) to[bend left=60] (7);
    \draw [line width=0.45mm, blue] (2) to[bend left=60] (6);
    \draw [line width=0.45mm, blue] (3) to[bend left=60] (5);
    
    \draw [line width=0.45mm, blue] (2) to[bend right=60] (3);
    \draw [line width=0.45mm, blue] (4) to[bend right=60] (400);
    \draw [line width=0.45mm, blue] (5) to[bend right=60] (500);
    \draw [line width=0.45mm, blue] (6) to[bend right=60] (600);
    \draw [line width=0.45mm, blue] (7) to[bend right=60] (700);

    \draw [line width=0.45mm, blue] (100) to[bend right=60] (700);
    \draw [line width=0.45mm, blue] (200) to[bend right=60] (600);
    \draw [line width=0.45mm, blue] (300) to[bend right=60] (500);
    \draw [line width=0.45mm, blue] (200) to[bend left=60] (300);

\end{tikzpicture}$$
\caption{The full meander of $\mathfrak{p}_{14}^C\frac{7}{1|2}$}
\label{fig:h2cmeander}
\end{figure}
\end{Ex}

\subsubsection{$\mathcal{H}_C({\color{blue}1,\dots,1},{\color{red}2})$}\label{sec:h2}
Let $\g=\mathfrak{p}_{2n}^C\frac{a_1|\dots|a_m}{b_1|\dots|b_t}$ have homotopy type $\mathcal{H}_C(\bn{1,\dots,1},\rn{2}),$ and denote by $\varphi_{(\rn{2})}$ the regular one-form on $\g$ constructed as in Section~\ref{sec:Cframework}. See Example~\ref{ex:aftertailfunctional}.

\begin{Ex}\label{ex:aftertailfunctional}
Consider the seaweed $\g=\mathfrak{p}_{16}^C\frac{2|3}{1|6},$ which has homotopy type $\mathcal{H}_C(\bn{1,1},\rn{2}).$ The regular one-form $\varphi_{(\rn{2})}$ has summands determined by the locations of black dots in Figure~\ref{fig:aftertailfunctional}.

\begin{figure}[H]
$$\begin{tikzpicture}[scale=0.45]
    \def\Node{\node [circle, fill, inner sep=1.5pt]}
    
    \draw (0,0)--(0,16)--(16,16)--(16,0)--(0,0);
    \draw[dotted] (0,0)--(16,16);
    \draw[dotted] (0,16)--(16,0);
    \draw[dotted] (0,8)--(16,8);
    \draw[dotted] (8,0)--(8,16);
    
    \draw [line width=.65mm, fill=blue, opacity=0.5] (0,16)--(0,14)--(2,14)--(2,11)--(5,11)--(5,5)--(11,5)--(11,2)--(14,2)--(14,0)--(16,0)--(16,1)--(15,1)--(15,7)--(7,7)--(7,15)--(1,15)--(1,16)--(0,16);
    \draw[line width=.65mm, fill=red, opacity=0.5] (7,7)--(7,9)--(9,9)--(9,7)--(7,7);
    
    \draw (0,15)--(1,15)--(1,14)--(6,14)--(6,6)--(14,6)--(14,1)--(15,1)--(15,0);
    
    \draw (4,13)--(5,13)--(5,12);
    \draw (13,4)--(13,5)--(12,5);
    
    \draw (2,15)--(2,14);
    \draw (3,14)--(3,13)--(2,13);
    \draw (3,13)--(3,12)--(4,12)--(4,13)--(3,13);
    \draw (4,11)--(4,12)--(5,12)--(5,11)--(6,11);
    \draw (5,10)--(7,10);
    \draw (6,9)--(7,9);
    \draw (9,7)--(9,6);
    \draw (10,5)--(10,7);
    \draw (11,4)--(12,4)--(12,5)--(11,5)--(11,6);
    \draw (13,3)--(12,3)--(12,4)--(13,4)--(13,3);
    \draw (14,3)--(13,3)--(13,2);
    \draw (15,2)--(14,2);
    
    \Node at (.5,14.5){};
    \Node at (6.5,14.5){};
    \Node at (2.5,11.5){};
    \Node at (5.5,13.5){};
    \Node at (4.5,12.5){};
    \Node at (5.5,5.5){};
    \Node at (6.5,6.5){};
    \Node at (7.5,8.5){};

\end{tikzpicture}$$
    \caption{The summands of $\varphi_{\rn{(2)}}$ identified within $\g$}
    \label{fig:aftertailfunctional}
\end{figure}

\noindent As Figure~\ref{fig:aftertailfunctional} displays, $\varphi_{(\rn{2})}$ is given by $$\varphi_{(\rn{2})}=e_{2,1}^*+e_{2,7}^*+e_{3,6}^*+e_{4,5}^*+e_{5,3}^*+e_{8,8}^*+e_{10,7}^*+e_{11,6}^*.$$
\end{Ex}

\begin{remark}
The summands of the regular one-form $\varphi_{(\rn{2})}$ can be visualized in the full meander of the associated seaweed $\g.$ Note that $v_n$ is the lone aftertail vertex when $\g$ has homotopy type $\mathcal{H}_C(\bn{1,\dots,1},\rn{2}).$  The regular one-form $\varphi_{(\rn{2})}$ is given by \begin{equation}\label{eq:fn2}
\varphi_{(\rn{2})}=e_{n,n}^*+\sum_{(v_i,v_j)}e_{i,j}^*,
\end{equation}
where the sum is over all edges $(v_i,v_j)$ in the oriented full meander $\overrightarrow{M}_{2n}^C(\g)$ such that either $i\leq n-1$ or $j\leq n-1.$
\end{remark}

\begin{theorem}\label{thm:aftertailkernel}
If $\g\subset\mathfrak{sp}(2n)$ has homotopy type $\mathcal{H}_C(\bn{1,\dots,1},\rn{2})$ with $\varphi_{(\rn{2})}$ constructed as in Section~\ref{sec:Cframework}, then $$\ker(B_{\varphi_{(\rn{2})}})=\text{span}\{h_{(\rn{2})}\},$$ where $$h_{(\rn{2})}=e_{n,n}-e_{n+1,n+1}.$$
\end{theorem}
\begin{proof}
Since $\varphi_{(\rn{2})}$ is regular on $\g,$ we need only show that $h_{(\rn{2})}\in\ker(B_{\varphi}).$ Consider the following:
\begin{itemize}
    \item $\varphi_{(\rn{2})}([h_{(\rn{2})},e_{i,i}-e_{2n-i+1,2n-i+1}])=0,$ for all $1\leq i\leq n,$
    \item $\varphi_{(\rn{2})}([h_{(\rn{2})},e_{i,j}\pm e_{2n-j+1,2n-i+1}])=0,$ for all $i,j$ such that $e_{i,j}^*$ is not a summand of $\varphi_{(\rn{2})},$
    \item $\varphi_{(\rn{2})}([h_{(\rn{2})},e_{i,j}- e_{2n-j+1,2n-i+1}])=0,$ for all summands $e_{i,j}^*$ of $\varphi_{(\rn{2})}$ such that $i,j\leq n-1,$ and
    \item $\varphi_{(\rn{2})}(h_{(\rn{2})},e_{i,j}])=0,$ for all summands $e_{i,j}^*$ of $\varphi_{(\rn{2})}$ such that either $i\leq n-1$ and $j>n,$ or $j\leq n-1$ and $i>n.$
\end{itemize}
The result follows.
\end{proof}

\begin{Ex}
Returning to the seaweed $\g=\mathfrak{p}_{16}^C\frac{2|3}{1|6}$ from Example~\ref{ex:aftertailfunctional}, we have that $$\ker(B_{\varphi_{(\rn{2})}})=\text{span}\{e_{8,8}-e_{9,9}\}.$$
\end{Ex}

\subsubsection{$\mathcal{H}_C(1,{\color{blue}1,\dots,1})$}\label{sec:h1}
Let $\g=\mathfrak{p}_{2n}^C\frac{a_1|\dots|a_m}{b_1|\dots|b_t}$ have homotopy type $\mathcal{H}_C(1,\bn{1,\dots,1})$ and denote by $\varphi_{(1)}$ the regular one-form on $\g$ constructed as in Section~\ref{sec:Cframework}. See Example~\ref{ex:outsidefunctional}.

\begin{Ex}\label{ex:outsidefunctional}
Consider the seaweed $\g=\mathfrak{p}_{12}^C\frac{3|2|1}{1|3},$ which has homotopy type $\mathcal{H}_C(1,\bn{1,1}).$ The regular one-form $\varphi_{(1)}$ has summands determined by the locations of black dots in Figure~\ref{fig:outsidefunctional}.

\begin{figure}[H]
$$\begin{tikzpicture}[scale=0.6]
    \def\Node{\node [circle, fill, inner sep=1.5pt]}
    
    \draw (0,0)--(12,0)--(12,12)--(0,12)--(0,0);
    \draw[dotted] (0,0)--(12,12);
    \draw[dotted] (0,12)--(12,0);
    \draw[dotted] (0,6)--(12,6);
    \draw[dotted] (6,0)--(6,12);
    
    \draw (1,12)--(1,10)--(3,10)--(3,9)--(0,9);
    \draw (12,1)--(10,1)--(10,3)--(9,3)--(9,0);
    \draw [line width=.65mm, fill=blue, opacity=0.5] (1,11)--(4,11)--(4,8)--(8,8)--(8,4)--(11,4)--(11,1)--(10,1)--(10,3)--(7,3)--(7,5)--(6,5)--(6,6)--(5,6)--(5,7)--(3,7)--(3,10)--(1,10)--(1,11);
    
    \draw (5,7)--(7,7)--(7,5);
    
    \draw (0,11)--(1,11);
    \draw (2,11)--(2,10);
    \draw (2,10)--(2,9);
    \draw (3,9)--(4,9);
    \draw (3,8)--(4,8);
    \draw (4,8)--(4,7);
    \draw (5,8)--(5,7);
    \draw (6,7)--(6,6);
    \draw (7,6)--(6,6);
    \draw (8,5)--(7,5);
    \draw (8,4)--(7,4);
    \draw (8,3)--(8,4);
    \draw (9,3)--(9,4);
    \draw (10,2)--(9,2);
    \draw (11,2)--(10,2);
    \draw (11,1)--(11,0);
    
    \Node at (0.5,11.5){};
    \Node at (3.5,10.5){};
    \Node at (0.5,9.5){};
    \Node at (2.5,9.5){};
    \Node at (3.5,7.5){};
    \Node at (7.5,7.5){};
    \Node at (6.5,6.5){};
\end{tikzpicture}$$
    \caption{The summands of $\varphi_{(1)}$ identified within $\g$}
    \label{fig:outsidefunctional}
\end{figure}
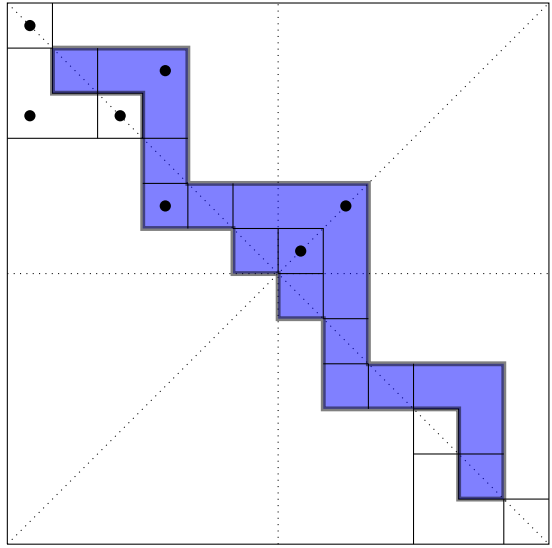
\noindent
As Figure~\ref{fig:outsidefunctional} displays, $\varphi_{(1)}$ is given by $$\varphi_{(1)}=e_{1,1}^*+e_{2,4}^*+e_{3,1}^*+e_{3,3}^*+e_{5,4}^*+e_{5,8}^*+e_{6,7}^*.$$
\end{Ex}

\begin{remark}\label{rem:outsidemeander}
The summands of the regular one-form $\varphi_{(1)}$ can be visualized in the meanders of the associated seaweed $\g.$ Since $\g$ has homotopy type $\mathcal{H}_C(1,\bn{1,\dots,1}),$ there is a unique path $P$ in $M_n^C(\g)$ with no endpoints in $T_n^C(\g).$ If $V(P)$ is the vertex set of $P,$ then $\varphi_{(1)}$ can be written as $$\varphi_{(1)}=\sum_{v_i\in V(P)}e_{i,i}^*+\sum_{(v_i,v_j)}e_{i,j}^*,$$ where the second sum is over all edges $(v_i,v_j)$ in the oriented full meander $\overrightarrow{M}_{2n}^C(\g)$ such that either $i\leq n$ or $j\leq n.$
\end{remark}

Utilizing Remark~\ref{rem:outsidemeander}, we arrive at a combinatorial description of the kernel of $B_{\varphi_{(1)}}$ in Theorem~\ref{thm:outsidekernel} below.

\begin{theorem}\label{thm:outsidekernel}
Let $\mathfrak{g}\subset\mathfrak{sp}(2n)$ have homotopy type $\mathcal{H}_C(1,\bn{1,\dots,1}).$ If $P$ is the unique path in the meander $M_n^C(\mathfrak{g})$ with zero endpoints in the tail and $\varphi_{(1)}$ is constructed as in Section~\ref{sec:Cframework}, then $$\ker(B_{\varphi_{(1)}})=\text{span}\{h_{(1)}\},$$ where $$h_{(1)}=\sum_{v_i\in V(P)}e_{i,i}-e_{2n-i+1,2n-i+1}.$$
\end{theorem}
\begin{proof}
Since $\varphi_{(1)}$ is regular on $\g,$ we need only show that $h_{(1)}\in\ker(B_{\varphi_{(1)}}).$ Consider the following:
\begin{itemize}
    \item $\varphi_{(1)}([h_{(1)},e_{i,i}-e_{2n-i+1,2n-i+1}])=0,$ for all $1\leq i\leq n,$
    \item $\varphi_{(1)}([h_{(1)},e_{i,j}\pm e_{2n-j+1,2n-i+1}])=0,$ for all $i,j$ such that $e_{i,j}^*$ is not a summand of $\varphi_{(1)},$
    \item $\varphi_{(1)}([h_{(1)},e_{i,j}- e_{2n-j+1,2n-i+1}])=\varphi_{(1)}(e_{i,j}- e_{2n-j+1,2n-i+1}-(e_{i,j}- e_{2n-j+1,2n-i+1}))=0,$ for all $i,j$ such that $e_{i,j}^*$ is a summand of $\varphi_{(1)}$ and $v_i,v_j\in V(P),$
    \item $\varphi_{(1)}([h_{(1)},e_{i,j}- e_{2n-j+1,2n-i+1}])=0,$ for all $i,j$ such that $e_{i,j}^*$ is a summand of $\varphi_{(1)},$ $i,j\leq n,$ and $v_i,v_j\not\in V(P),$ and
    \item $\varphi_{(1)}([h_{(1)},e_{i,j}])=0,$ for all $i,j$ such that $e_{i,j}^*$ is a summand of $\varphi_{(1)}$ and $i+j=2n+1.$
\end{itemize}
The result follows.
\end{proof}

\begin{Ex}
Returning to the seaweed $\g=\mathfrak{p}_{12}^C\frac{3|2|1}{1|3}$ from Example~\ref{ex:outsidefunctional}, note that $P=\{v_1,v_3\}$ is the unique path in $M_n^C(\mathfrak{g})$ with zero endpoints in the tail \textup(see Figure~\ref{fig:outsidemeander}\textup). Therefore, we have that $$\ker(B_{\varphi_{(1)}})=\text{span}\{e_{1,1}-e_{12,12}+e_{3,3}-e_{10,10}\}.$$

\begin{figure}[H]
$$\begin{tikzpicture}[scale=0.7]
    \def\Node{\node [circle, fill, inner sep=1.5pt]}
    \vertex (1) at (1,0){1};
    \vertex (2) at (2,0){2};
    \vertex (3) at (3,0){3};
    \vertex (4) at (4,0){4};
    \vertex[fill=blue, fill opacity=0.5, text opacity=1] (5) at (5,0){5};
    \vertex[fill=blue, fill opacity=0.5, text opacity=1] (6) at (6,0){6};
    \draw (1) to[bend left=60] (3);
    \draw (4) to[bend left=60] (5);
    
    \draw [line width=0.45mm, blue] (2) to[bend right=60] (4);

\end{tikzpicture}$$
\caption{The meander of $\mathfrak{p}_{14}^C\frac{3|2|1}{1|3}$}
\label{fig:outsidemeander}
\end{figure}
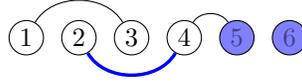
\end{Ex}

\section{Main results}\label{sec:main}
In this section, we establish the main result of this article. In particular, we prove that all type-C seaweeds of index one are contact by showing that each of the one-forms discussed in Sections~\ref{sec:h3}, \ref{sec:h2}, and \ref{sec:h1} are contact. We accomplish this by utilizing the following lemma, which is a classical characterization of contact Lie algebras in terms of an associated regular one-form $\varphi$ and the kernel of the associated $B_{\varphi}.$

%Let $\mathfrak{f}$ be a $(2k+1)-$dimensional contact Lie algebra with ordered basis $\mathscr{B}=\{E_1,\dots,E_{2k+1}\}$ and contact form $\varphi=\sum_{i=1}^{2k+1}x_iE_i^*,$ where $\{E_1^*,\dots,E_{2k+1}^*\}$ is the basis of $\g^*$ that is dual to $\mathscr{B}.$ If $[B_{\varphi}]$ is the matrix corresponding to the Kirillov form, i.e., $[B_{\varphi}]=[\varphi([E_i,E_j])]_{1\leq i,j\leq 2k+1},$ then define $$\left[\widehat{B}_{\varphi}\right]=\begin{bmatrix}
%0 & [\varphi]^t\\
%-[\varphi] & [B_{\varphi}]
%\end{bmatrix},$$ where %$[\varphi]^t=\begin{bmatrix}
%x_1 & \dots & x_{2k+1}
%\end{bmatrix}$. It is straightforward to see that $$\varphi\wedge(d\varphi)^k=\det\left(\left[\widehat{B}_{\varphi}\right]\right)E_1^*\wedge\dots\wedge E_{2k+1}^*$$ (see \textbf{\cite{Sally}}), and so we have the following lemma.

\begin{lemma}\label{lem:kernel}
Let $\mathfrak{f}$ be an arbitrary Lie algebra. If $\ind\mathfrak{f}=1,$ then a regular one-form $\varphi\in\mathfrak{f}^*$ with $\ker(B_{\varphi})=\text{span}\{x\}$ is contact if and only if $\varphi(x)\neq 0.$
\end{lemma}

We are now in a position to prove the main theorem of this article.

\begin{theorem}\label{thm:main}
If $\g$ is a type-A or type-C seaweed, then $\g$ is contact if and only if $\ind \g=1.$
\end{theorem}

\begin{proof}
If $\g$ is a type-A seaweed, then by Theorem 24 of \textbf{\cite{contactA}}, $\g$ is contact if and only if $\ind\g=1.$ So let $\mathfrak{g}$ be an index-one seaweed of type C. Recall from Corollary~\ref{cor:ind1htC} that this means $\mathfrak{g}$ has one of the following homotopy types:
\begin{enumerate}
    \item[1a.] $\mathcal{H}_C(\bn{1,\dots,1,2,1,\dots,1}),$
    \item[1b.] $\mathcal{H}_C(\bn{1,\dots,1,3,1,\dots,1})$,
    \item[2.] $\mathcal{H}_C(1,{\color{blue}1,\dots,1})$, or
    \item[3.] $\mathcal{H}_C({\color{blue}1,\dots,1},{\color{red}2})$.
\end{enumerate}
We proceed by treating each homotopy type as its own case, using the constructions of Sections~\ref{sec:h3}, \ref{sec:h2}, and \ref{sec:h1}, and then applying Lemma~\ref{lem:kernel}.

\medskip
\noindent

\medskip
\noindent

\medskip
\noindent
\textbf{Case 1:} The seaweed $\g$ has homotopy type $\mathcal{H}_C(\bn{1,\dots,1,c,1,\dots,1}),$ where $\bn{c}=\bn{2}$ or $\bn{3}.$ We claim that $\varphi_{(\bn{c})}$ is a contact form on $\g,$ for $\bn{c}=\bn{2}$ or $\bn{3},$ respectively. Recall from Theorem~\ref{thm:h2C} that if $\mathscr{C}=(v_{i_1},v_{i_2},\dots,v_{i_k},v_{i_1})$ is the unique cycle in the full meander $M_{2n}^C(\g),$ then $$\ker(B_{\varphi_{(\bn{c})}})=\text{span}\{h_{(\bn{c})}\}=\text{span}\left\{\sum_{j=1}^k(-1)^{j+1}e_{i_j,i_j}\right\}.$$ By Lemma~\ref{lem:oppsignC} and the fact that $\varphi_{(\bn{c})}$ only contains summands $e_{i,j}^*$ such that $i+j\leq2n+1,$ we have that $\varphi_{(\bn{c})}(h_{(\bn{c})})=\frac{k}{4}\neq 0,$ for $\bn{c}=\bn{2}$ or $\bn{3}.$ An application of Lemma~\ref{lem:kernel} establishes the claim.

\medskip
\noindent
\textbf{Case 2:} The seaweed $\g$ has homotopy type $\mathcal{H}_C(\bn{1,\dots,1},\rn{2}).$ We claim that $\varphi_{(\rn{2})}$ is a contact form on $\g.$ Recall from Theorem~\ref{thm:aftertailkernel} that $$\ker(B_{\varphi_{(\rn{2})}})=\text{span}\{h_{(\rn{2})}\}=\text{span}\{e_{n,n}-e_{n+1,n+1}\},$$ and notice that $\varphi_{(\rn{2})}(h_{(\rn{2})})=1\neq 0.$ An application of Lemma~\ref{lem:kernel} establishes the claim.

\medskip
\noindent
\textbf{Case 3:} The seaweed $\g$ has homotopy type $\mathcal{H}_C(1,\bn{1,\dots,1}).$ We claim that $\varphi_{(1)}$ is a contact form on $\g.$ Recall from Theorem~\ref{thm:outsidekernel} that if $P$ is the unique path in the meander $M_n^C(\g)$ with zero vertices in the tail, then $$\ker(B_{\varphi_{(1)}})=\text{span}\{h_{(1)}\}=\text{span}\left\{\sum_{v_i\in V(P)}e_{i,i}-e_{2n-i+1,2n-i+1}\right\},$$ and notice that $\varphi_{(1)}(h_{(1)})=|V(P)|\neq 0.$ An application of Lemma~\ref{lem:kernel} establishes the claim.
\end{proof}

\section{Examples}\label{sec:examples}
For type-C seaweeds, Theorem~\ref{thm:main} simplifies the question about the existence of a contact form on a seaweed $\g$ to one about $\g$'s index. But the index may be computed from $\g$'s meander (see Theorem~\ref{thm:indexC}).  This is an elegant reduction, but some burdensome complexity remains, for one must first construct the meander and then count the number and type of connected components.  Fortunately, as with type-A, for type-C seaweeds defined by (partial) compositions consisting of a small number of parts, we can quickly determine the index of $\mathfrak{g}$.
\begin{theorem}[Coll et al. \textbf{\cite{indC}}, 2017]\label{thm:3parts}
Let $\alpha+\beta=n$. If $\gamma=n-1$ or $\gamma=n-2$, then 
$$\ind \mathfrak{p}_{2n}^C \frac{\alpha|\beta}{\gamma}= \gcd(\alpha+\beta,\beta+\gamma)-1.$$
\end{theorem}

\noindent
Making use of Theorem~\ref{thm:3parts}, we can construct a spate of contact Lie algebras -- we need only ensure that the seaweed $\mathfrak{p}_{2n}^C\frac{\alpha|\beta}{\gamma}$ satisfies $\gamma=n-1$ or $\gamma=n-2$ and $\gcd(\alpha+\beta,\beta+\gamma)=2.$ This is easy to do: For $\gamma=n-2$, consider $\mathfrak{p}_{18}^C\frac{5|3}{7},$ and for $\gamma=n-1$, consider $\mathfrak{p}_{14}^C\frac{4|2}{6},$  or the much larger $\mathfrak{p}_{100}^C\frac{15|13}{49}$.  The reader will have no difficulty in developing others.

%\section{Classification Result -- An Announcement }\label{sec:announcement}

%$\newfactor{A}{B}$

%In a forthcoming note, and using existential -- rather than constructive -- methods, we establish the following complete classification of contact seaweeds.

%\begin{theorem}
%An index-one seaweed subalgebra of a reductive Lie algebra is contact if and only if it is quasi-reductive. \end{theorem}

%A finite-dimensional Lie algebra $\g$ is \textit{quasi-reductive} if it admits a form $\varphi\in \g^*$ such that 
%$\faktor{\ker(B_\varphi)}{Z(g)}$ is a reductive Lie algebra whose center consists of semisimple elements of $\g$. Here, $B_\varphi$ is the Kirillov form of $\varphi$ and $Z(\g)$ is the center of $\g$.

%Recall also that in \textbf{\cite{PanRais}}, Panyushev established, in particular, that all Type-A and Type-C seaweeds are ``quasi-reductive".  

%\begin{theorem}[Elashvili \textbf{\cite{Elashvili}}, 1990]\label{thm:elashvili}
%The type-$A$ seaweed $\mathfrak{p}_n^A\frac{a|c}{n}$ %has index $\gcd(a,c)-1.$
%\end{theorem}

\end{document}